\documentclass[12pt,twoside,a4paper]{article}
\usepackage[sectionbib,round]{natbib}
\usepackage{amsmath,amssymb,amsfonts,amsthm,xpatch}
\usepackage{fancyhdr,bm,enumitem,hyperref,graphicx}
\usepackage{multirow,tabularx,authblk,booktabs}
\usepackage{sectsty,secdot,setspace}
\usepackage{algorithm,algorithmic,graphicx}
\usepackage[top = 3cm, bottom = 2.5cm, left= 3cm, right= 3cm]{geometry}
\usepackage[tableposition = top]{caption}
\newcommand{\RNum}[1]{\uppercase\expandafter{\romannumeral #1\relax}}
\newcolumntype{Y}{>{\centering\arraybackslash}X}
\DeclareMathOperator*{\argmax}{arg\,max}
\DeclareMathOperator*{\argmin}{arg\,min}
\newtheorem{thm}{Theorem}

\numberwithin{exm}{section}
\newtheorem{cor}{Corollary}

\newtheorem{rmk}{Remark}
\newtheorem{prop}{Proposition}
\theoremstyle{definition}
\sectionfont{\fontsize{14pt}{14pt}\selectfont}
\subsectionfont{\fontsize{12pt}{12pt}\selectfont}
\def\spacingset#1{\renewcommand{\baselinestretch}%
{#1}\small\normalsize} \spacingset{1}
\spacingset{1.5}
\fancypagestyle{plain}{%
\fancyhf{} 

}
\title{\large On the sure screening properties of iteratively sure independence screening algorithms }
\author[1]{\normalsize Ning Zhang \thanks{Email: nzhang2018@gmail.com}}
\author[2, 1]{\normalsize Wenxin Jiang \thanks{Email: wjiang@northwestern.edu}}
\author[3, 2]{\normalsize Yuting Lan \thanks{Email: lan.yuting@mail.shufe.edu.cn}}
\affil[1]{\normalsize Shandong University\vspace{1ex}}
\affil[2]{\normalsize Northwestern University\vspace{1ex}}
\affil[3]{\normalsize Shanghai University of Finance and Economics}
\date{}
\begin{document}
\maketitle
\noindent\makebox[\textwidth][c]{%
\begin{minipage}{0.9\textwidth}
\spacingset{1.4}
\rule{\textwidth}{0.4pt}
{\bf{Abstract}}\\
\cite{sis} proposed the path-breaking theory of sure independence screening (SIS) and an iterative algorithm (ISIS) to effectively reduce the predictor dimension for further variable selection approaches. \cite{vanisis} extended ISIS to generalized linear models and introduced the Vanilla ISIS (Van-ISIS) algorithm, allowing selected predictors to be screened out in upcoming iterations. The success of SIS depends on its sure screening property, which was obtained by \cite{sis} under the marginal correlation assumption. However, despite wide applications of ISIS and Van-ISIS in various scientific fields, their sure screening properties have not been proved during the past decade. To fill this gap, we prove the sure screening properties of three different types of iterative algorithms for linear models without relying on the marginal correlation assumption, where ISIS and Van-ISIS can be regarded as two special cases of them.
\rule{\textwidth}{0.4pt}
\end{minipage}}
\newpage
\section{Introduction}\label{sec:intro}

In the big data era, scientists are confronted with unprecedentedly massive and complex data in various fields, such as genomics, finance and earth sciences, etc. For high dimensional data with myriads of predictors, often very few of them are believed to be truly relevant to the response. Thus, how to extract key information and identify relevant predictors from high dimensional datasets becomes a great challenge for statisticians.

The past decade has witnessed an explosion in the development of variable screening techniques, which are designed to efficiently reduce the predictor dimension to a manageable size so that variable selection approaches, such as the LASSO \citep{lasso}, the SCAD \citep{scad}, the adaptive LASSO \citep{adalasso}, the elastic net \citep{elanet} and others, can be implemented smoothly afterwards to identify relevant predictors. The advance in variable screening techniques can be traced back to 2008 when \cite{sis} proposed the seminal sure independence screening (SIS) method. SIS could efficiently conduct the dimension reduction through ranking marginal correlations between predictors and the response. For ultrahigh dimensional data, applying SIS before variable selection techniques can significantly lower the computational cost of directly solving large-scale optimization problems. Most importantly, SIS could preserve all relevant predictors in the screening process with an overwhelming probability under certain assumptions, which is referred to as the sure screening property \citep{sis}.

The sure screening property is a main consideration when designing variable screening techniques since it ensures that all relevant predictors can be preserved with an overwhelming probability for upcoming variable selection procedures. Nevertheless, the sure screening property of SIS relies on the marginal correlation assumption, requiring marginal correlations between relevant predictors and the response to be bounded away from zero. Consequently, as pointed out by \cite{sis}, SIS may break down when any relevant predictor is jointly correlated but marginally uncorrelated with the response, or many irrelevant predictors have higher marginal correlations with the response than some relevant ones do. 

To avoid such undesirable results, \cite{sis} introduced the iteratively sure independence screening (ISIS) method, which iteratively employs SIS on remaining predictors and the residual vector obtained from the regression between the response and selected predictors to select candidate predictors. Then some variable selection approach, such as LASSO or SCAD, can be applied on those candidate predictors to determine which ones to be added to the selected model. Moreover, \cite{vanisis} extended ISIS to generalized linear models and proposed the Vanilla ISIS (Van-ISIS) algorithm, allowing predictors selected in previous steps to be screened out in the upcoming iterations. \cite{isisr} slightly modified the Van-ISIS algorithm and implemented it in the R package ``SIS'' to facilitate its applications in various fields.

Despite the outstanding performances of these iterative approaches in both simulation studies \citep{vanisis} and real data analyses \citep{isisapp}, their sure screening properties have not been theoretically verified during the past decade. In previous literatures, asymptotic properties of the variable selection methods and the variable screening techniques are always investigated separately. However, it is necessary to
study them simultaneously for both ISIS and Van-ISIS since each iteration of the algorithms can be regarded as a two-stage procedure, where some variable selection technique is applied on candidate predictors obtained from the screening results. Furthermore, the proof of the sure screening property for Van-ISIS can be even more challenging considering the fact that relevant predictors obtained in previous iterations can be dropped from the models selected afterwards.

To overcome these challenges, inspired by \cite{fr}'s proof of the screening consistency of forward regression (FR), we begin with proving the sure screening properties of a type of simplified iterative screening algorithms, where no variable selection approach is applied in each iteration. FR also falls into this category and our proof could lead to a sharper result than that in \cite{fr}. In the next, applying similar techniques, we prove the sure screening properties of the other two types of iterative algorithms with some variable selection method employed in each iteration. Consequently, the sure screening properties of ISIS and Van-ISIS can be achieved directly as special cases of these two types of algorithms.

The rest of the paper is organized as follows. In Section 2, we review the ISIS and Van-ISIS methods in detail and introduce the three types of iterative algorithms considered in our main theorems. We then list required assumptions for our theoretical results and formally describe the sure screening properties of these three types of algorithms in Section 3. Next, we present some preliminary results and prove our main theorems in Section 4, where the detailed proofs of preliminary results can be referred to the Appendix. Finally, we briefly summarize our results and discuss potential work in the future.

\section{Background}
In this section, we introduce necessary notation and review the SIS-based iterative screening methods for linear models, including ISIS \citep{sis}, Van-ISIS \citep{vanisis} and one of its variants \citep{isisr}. And three types of iterative algorithms are introduced at the end of this section as generalizations of them.
\subsection{Models and notation}
Throughout the paper, we consider the classic linear model
\[y=\bm{x}^\top \bm{\beta}+\epsilon,\]
where $y$ denotes the response, $\bm{x}=(x_1,\cdots,x_p)^\top $ denotes the predictor vector, $\bm{\beta}=(\beta_1,\cdots,\beta_p)^\top $ denotes the regression coefficient vector and $\epsilon$ denotes the random error. With $n$ realizations of $y$ and $\bm{x}$, the model can be written as
\[{Y}={X}\bm{\beta}+\bm{\epsilon},\]
where ${Y}=(Y_1,\cdots,Y_n)^\top  \in \mathbb{R}^n$ is the response vector, ${X}=[{X}_1,\cdots,{X}_p]\in \mathbb{R}^{n\times p}$ denotes the design matrix and $\bm{\epsilon}=(\epsilon_1,\cdots,\epsilon_n)^\top  \in \mathbb{R}^n$ consists of $n$ i.i.d random errors. Additionally, we denote $\mathcal{T}=\{j:\beta_j\neq 0\}$ as the true model of size $|\mathcal{T}|=t$, including indices of all the relevant predictors. 

For any index set $\mathcal S\subset\{1,\cdots,p\}$, let $\bm{\beta}_{\mathcal {S}}$ denote the subvector of $\bm{\beta}$ consisting of the $j$-th entry in $\bm{\beta}$ with $j\in {\mathcal S}$ and ${X}_{\mathcal S}$ denote the submatrix of ${X}$ with columns corresponding to ${\mathcal S}$. Furthermore, denote $C({X}_{\mathcal S})$ as the linear space spanned by columns of ${X}_{\mathcal S}$ and $C({X}_{\mathcal S})^\perp$ as its orthogonal complement. Then, if ${X}_{\mathcal S}$ is of full column rank, the orthogonal projection matrix on $C({X}_{\mathcal S})$ can be expressed as ${H}_{\mathcal S}={X}_{\mathcal S}({X}^\top _{\mathcal S}{X}_{\mathcal S})^{-1}{X}^\top _{\mathcal S}$ and ${M}_{\mathcal S}={I}_n-{H}_{\mathcal S}$ represents the orthogonal projection matrix on $C({X}_{\mathcal S})^\perp$ with $I_n$ denoting the $n\times n$ identity matrix.

\subsection{Iterative screening algorithms}
Initially, we review the ISIS algorithm proposed by \cite{sis}, which works as follows.

\begin{enumerate}[label={\em Step \arabic*.},align=left, leftmargin=*]
    \item Select the model $\mathcal{A}_1$ of size $a_1$ as
    \[\mathcal{A}_1=\left\{1\leq i \leq p:\,\,|X_i^\top Y| \text{ is among the first } a_1 \text{ largest of all}\right\}.\]
    Then obtain the submodel $\mathcal{B}_1$ from $\mathcal{A}_1$ through minimizing the penalized least squares (PLS) as
    \[\min_{\bm{\beta}\in\mathbb{R}^{{\scriptstyle a}_{\scalebox{.7}{$\scriptscriptstyle 1$}}}}\left[\big|\big|Y-X_{\mathcal{A}_1}\bm{\beta}\big|\big|^2+n\sum_{j=1}^{a_1}p_{\lambda_n}(|\beta_j|)\right],\]
    where $||\cdot||$ denotes the $L_2$ norm of vectors and $p_{\lambda_n}(\cdot)$ is some penalty function with a tuning parameter $\lambda_n>0$. Define  $\mathcal{S}_1=\mathcal{B}_1$ and the remaining model as $\mathcal{S}_1^c=\{1,\cdots,p\}-\mathcal{S}_1$.
    \item Based on the selected model $\mathcal{S}_k$ of size $s_k$, select the model $\mathcal{A}_{k+1}$ of size $a_{k+1}$ as
    \[\mathcal{A}_{k+1}=\left\{i \in\mathcal{S}_k^c:\,\,|X_i^\top M_{\mathcal{S}_k}Y| \text{ is among the first } a_{k+1} \text{ largest of all}\right\},\]
    where $M_{\mathcal{S}_k}Y$ denotes the residual vector from regressing $Y$ over $X_{\mathcal{S}_k}$. The submodel $\mathcal{B}_{k+1}$ is obtained from ${\mathcal{A}_{k+1}}$ by minimizing PLS between the residual vector and $X_{\mathcal{A}_{k+1}}$ as
    \[\min_{\bm{\beta}\in\mathbb{R}^{{\scriptstyle a}_{\scalebox{.7}{$\scriptscriptstyle k+1$}}}}\left[\big|\big|M_{\mathcal{S}_k}Y-X_{\mathcal{A}_{k+1}}\bm{\beta}\big|\big|^2+n\sum_{j=1}^{a_{k+1}}p_{\lambda_n}(|\beta_j|)\right].\]
    Set $\mathcal{S}_{k+1}=\mathcal{S}_{k}\cup\mathcal{B}_{k+1}$ and $\mathcal{S}_{k+1}^c=\{1,\cdots,p\}-\mathcal{S}_{k+1}$.
    \item Iterate Step 2 until we obtain the model $\mathcal{S}_k$ with $k=\kappa$ for some predetermined maximum number of iterations $\kappa$.
\end{enumerate}

In the ISIS algorithm, predictors selected in each iteration are included in the final model. However, it is not the case for the Van-ISIS \citep{vanisis} approach, which operates as follows.
\begin{enumerate}[label={\em Step \arabic*.}, align=left, leftmargin=*]
    \item Select the model $\mathcal{A}_1$ of size $a_1$ as
    \[\mathcal{A}_1=\left\{1\leq i \leq p:\,\,||M_iY||^2 \text{ is among the first } a_1 \text{ smallest of all}\right\},\]
    that is, model $\mathcal{A}_1$ corresponds to predictors that lead to the smallest residual sums of squares (RSS) in the componentwise regression with the response. Then model $\mathcal{S}_1$ is obtained through solving the PLS problem
    \[\min_{\bm{\beta}\in\mathbb{R}^{{\scriptstyle a}_{\scalebox{.7}{$\scriptscriptstyle 1$}}}}\left[\big|\big|Y-X_{\mathcal{A}_{1}}\bm{\beta}\big|\big|^2+n\sum_{j=1}^{a_{1}}p_{\lambda_n}(|\beta_j|)\right].\]
    
    \item Based on the model $\mathcal{S}_k$ of size $s_k$, we choose the model $\mathcal{A}_{k+1}$ as
    \[\mathcal{A}_{k+1}=\left\{i \in\mathcal{S}_k^c:\,\,||M_{\mathcal{S}_k\cup\{i\}}Y||^2 \text{ is among the first } a_{k+1} \text{ smallest}\right\},\]
    where $M_{\mathcal{S}_k\cup\{i\}}Y$ denotes the residual obtained from regressing $Y$ over $X_{\mathcal{S}_k}$ and $X_i$ for each $i\in\mathcal{S}_k^c$. The model $\mathcal{S}_{k+1}$ is determined by the variable selection procedure on $X_{\mathcal{S}_k\cup\mathcal{A}_{k+1}}=[X_{\mathcal{S}_k},X_{\mathcal{A}_{k+1}}]$ as
    \[\min_{\bm{\beta}\in\mathbb{R}^{{\scriptstyle s}_{\scalebox{.7}{$\scriptscriptstyle k$}}+{\scriptstyle a}_{\scalebox{.7}{$\scriptscriptstyle k+1$}}}}\left[\big|\big|Y-X_{\mathcal{S}_k\cup\mathcal{A}_{k+1}}\bm{\beta}\big|\big|^2+n\sum_{j=1}^{s_k+a_{k+1}}p_{\lambda_n}(|\beta_j|)\right].\]
    \item Iterate Step 2 until we obtain the model $\mathcal{S}_k$ with $k=\kappa$ for some predetermined maximum number of iterations $\kappa$.
\end{enumerate}

\cite{isisr} slightly modified the Van-ISIS algorithm and implemented it in the R package ``SIS''. We denote the modified algorithm as Van-ISIS-R and the only difference between Van-ISIS and Van-ISIS-R arises in the criterion of selecting $\mathcal{A}_{k+1}$. Instead of choosing the predictors that minimize the RSS, Van-ISIS-R computes 
\begin{equation}\label{rbetaj}
(\hat{\bm{\beta}}_{\mathcal{S}_k},\hat\beta_i)=\argmin_{{\bm{\beta}}_{\mathcal{S}_k}\in\mathbb{R}^{{\scriptstyle s}_{\scalebox{.7}{$\scriptscriptstyle k$}}},\beta\in\mathbb{R}}\big|\big|Y-X_{\mathcal{S}_k}\bm{\beta}_{\mathcal{S}_k}-X_i\beta\big|\big|^2,
\end{equation}
and determines $\mathcal{A}_{k+1}$ as 
\[\mathcal{A}_{k+1}=\left\{i \in\mathcal{S}_k^c:\,\,\big|\hat\beta_i\big| \text{ is among the first } a_{k+1} \text{ largest of all } \big|\hat\beta_i\big|s\right\}.\]

Each iteration of aforementioned iterative algorithms can be regarded as a two-stage variable selection procedure, where some variable selection method is applied after the set of candidate predictors is determined. To investigate their sure screening properties, we begin with considering a type of simplified iterative algorithms with no variable selection method applied, including the non-penalized versions of ISIS, Van-ISIS and Van-ISIS-R, where the new model $\mathcal{S}_{k+1}$ is determined as $\mathcal{S}_{k+1}=\mathcal{S}_{k}\cup\mathcal{A}_{k+1}$ in each iteration of these three algorithms. It is also noteworthy that FR is equivalent to the non-penalized version of Van-ISIS with $a_k=1$ for $k\geq 1$.

After scrutinizing these penalized and their non-penalized versions, we see that each of them can be regarded as a combination of some screening procedure determining $\mathcal{A}_{k+1}$ based on $\mathcal{S}_{k}$, and certain selection procedure choosing $\mathcal{S}_{k+1}$ from $\mathcal{A}_{k+1}$ and $\mathcal{S}_{k}$. For the screening procedure, we can choose among the following three criteria.
\begin{enumerate}[align=left, leftmargin=*]
\item $\mathcal{A}_{k+1}=\{i \in\mathcal{S}_k^c:\,\,|X_i^\top M_{\mathcal{S}_k}Y| \text{ is among the first } a_{k+1} \text{ largest of all}\}$.
\item $\mathcal{A}_{k+1}=\{i \in\mathcal{S}_k^c:\,\,||M_{\mathcal{S}_k\cup\{i\}}Y||^2 \text{ is among the first } a_{k+1} \text{ smallest of all}\}$.
\item $\mathcal{A}_{k+1}=\{i \in\mathcal{S}_k^c:\,\,|\hat\beta_i| \text{ is among the first } a_{k+1} \text{ largest of all}\}$, where $\hat\beta_i$ is computed as in equation \eqref{rbetaj}.
\end{enumerate}
For the selection procedure, we also have the following three choices.
\begin{enumerate}[align=left, leftmargin=*]
\item $\mathcal{S}_{k+1}=\mathcal{S}_{k}\cup\mathcal{A}_{k+1}$.
\item $\mathcal{S}_{k+1}=\mathcal{S}_{k}\cup\mathcal{B}_{k+1}$, where $\mathcal{B}_{k+1}$ is obtained by solving
\[\min_{\bm{\beta}\in\mathbb{R}^{{\scriptstyle a}_{\scalebox{.7}{$\scriptscriptstyle k+1$}}}}\left[\big|\big|M_{\mathcal{S}_k}Y-X_{\mathcal{A}_{k+1}}\bm{\beta}\big|\big|^2+n\sum_{j=1}^{a_{k+1}}p_{\lambda_n}(|\beta_j|)\right].\]
\item $\mathcal{S}_{k+1}$ is determined by solving
    \[\min_{\bm{\beta}\in\mathbb{R}^{{\scriptstyle s}_{\scalebox{.7}{$\scriptscriptstyle k$}}+{\scriptstyle a}_{\scalebox{.7}{$\scriptscriptstyle k+1$}}}}\left[\big|\big|Y-X_{\mathcal{S}_k\cup\mathcal{A}_{k+1}}\bm{\beta}\big|\big|^2+n\sum_{j=1}^{s_k+a_{k+1}}p_{\lambda_n}(|\beta_j|)\right].\]
\end{enumerate}

For simplicity, we only consider LASSO and SCAD in the selection procedure throughout the paper. LASSO employs the $L_1$ penalty with the form
\begin{equation}\label{p.lasso}
p_{\lambda_n}(\theta)=\lambda_n|\theta|.
\end{equation}
And the derivative of the SCAD penalty function is given by
\begin{equation}\label{p.scad}
p'_{\lambda_n}(\theta)=\lambda_n\left\{I(\theta\leq \lambda_n)+\frac{(a\lambda_n-\theta)_{+}}{(a-1)\lambda_n}I(\theta> \lambda_n)\right\}\quad \text{for some } a>2,
\end{equation}
where $p_{\lambda_n}(0)=0$ and $a$ is often set to $3.7$.

Therefore, there are totally nine combinations of the screening and selection criteria, which could cover all aforementioned iterative algorithms. For instance, the non-penalized version of ISIS is equivalent to the algorithm that applies the Screening Criterion 1 and the Selection Criterion 1, which can be denoted as SCR1-SEL1. Similarly, ISIS is equivalent to SCR1-SEL2, Van-ISIS can be written as SCR2-SEL3,  Van-ISIS-R is the same as SCR3-SEL3 and FR can be regarded as SCR2-SEL1 with $a_k=1$ for $k\geq 1$. In the rest of the paper, we consider the sure screening properties of the three types of iterative algorithms corresponding to the three different selection criteria. 

\section{Sure screening properties of iterative algorithms} 
In this section, we introduce necessary assumptions for our theoretical results and formally describe the sure screening properties of three types of iterative algorithms in three theorems.
\subsection{Technical assumptions}
Our main theorems rely on the following four technical assumptions.
\begin{enumerate}[label=({A\arabic*}),align=left, leftmargin=*]
    \item There exist positive constants $c_t$, $c_p$, $c_\beta$ and $c_y$, together with $\xi_t,\xi_y\geq 0$ and $\xi_p,\xi_\beta>0$ satisfying that $\xi_p+3\xi_\lambda<1$ with $\xi_\lambda=\xi_t+\xi_y+2\xi_{\beta}$, such that
    \[t=|\mathcal T|\leq c_tn^{\xi_t},\quad \log n < \log p\leq c_pn^{\xi_p},\]
    and
    \[\beta_{\text{min}}=\min_{i\in\mathcal{T}}|\beta_i|\geq c_{\beta}n^{-\xi_{\beta}},\quad \text{var}(y)=\sigma_y^2\leq c_y n^{\xi_y}.\]
    \item The distribution of $\bm{x}$ satisfies that, for any constant $s=O(n^{\xi_s})$ with some $\xi_s$ satisfying that $\xi_p+3\xi_s<1$, there exists some positive constant $c_s$ such that
    \begin{align*}
    P\Big[\tau_{\text{min}}\leq \min\limits_{|\mathcal{S}|\leq s}\lambda_{\text{min}}\{\hat{{\Sigma}}_{(\mathcal{S})}\}&\leq \max\limits_{|\mathcal{S}|\leq s}\lambda_{\text{max}}\{\hat{{\Sigma}}_{(\mathcal{S})}\}\leq\tau_{\text{max}}\Big]\\
    &\geq 1 - O\left(\exp\left(-c_s n^{1-2\xi_s}\right)\right),
    \end{align*}
    where $\hat{{\Sigma}}_{(\mathcal{S})}={X}_{\mathcal{S}}^\top {X}_{\mathcal{S}}/n$, $\lambda_{\text{min}}\{\hat{{\Sigma}}_{(\mathcal{S})}\}$ and $\lambda_{\text{max}}\{\hat{{\Sigma}}_{(\mathcal{S})}\}$ denote the smallest and largest eigenvalues of $\hat{{\Sigma}}_{(\mathcal{S})}$ respectively, and $\tau_{\text{min}}$ and $\tau_{\text{max}}$ are some constants satisfying $0<\tau_{\text{min}}<\tau_{\text{max}}<\infty$.
    \item The random error $\epsilon$ is independent of $\bm{x}$ and follows a sub-Gaussian distribution with zero mean and finite variance $\sigma^2$.
    \item The response $y$ follows a mean-zero distribution satisfying that there exist some positive constants $c_Y$ and $\xi_Y$, such that
    \[P\left(\frac{1}{n}\sum_{i=1}^nY_i^2\geq2\sigma_y^2\right)\leq O\left(\exp\left(-c_Yn^{\xi_Y}\right)\right).\]
\end{enumerate}

In assumption (A1), we set restrictions on the size of true model and the number of predictors, which coincides with the sparse model assumption and allows the predictor dimension to increase exponentially with the sample size. The assumption on $\beta_{\text{min}}$ was adopted in various literatures \citep{sis,holp,fr} to prevent non-zero coefficients from converging to zero too fast such that they can be identified consistently. And we also allow the variance of the response to diverge with the sample size.

\cite{fr} proved that assumption (A2) holds when $\bm{x}\sim N(\bm{0},\Sigma)$ and the eigenvalues of $\Sigma$ satisfy that 
\begin{equation}\label{covm}
2\tau_{\text{min}}\leq \lambda_{\text{min}}(\Sigma)\leq  \lambda_{\text{max}}(\Sigma)\leq 0.5\tau_{\text{max}}.
\end{equation}
Moreover, applying similar techniques and the Hoeffding's inequality \citep{hoeffding}, we can prove assumption (A2) for bounded predictors with the covariance matrix satisfying condition \eqref{covm}. 

Under assumption (A3), we have the following result for the weighted sum of random errors.
\begin{prop}[\citeauthor{ver}, \citeyear{ver}, Proposition 5.10]\label{prop:tail}
Suppose that $\epsilon$ follows a mean-zero sub-Gaussian distribution and let $\{\epsilon_i\}_{i=1}^n$ be $n$ independent realizations of $\epsilon$. Then for any $\bm{v}=(v_1,\cdots,v_n)\in \mathbb{R}^n$ with $||\bm{v}||^2=1$ and any $z>0$, we have
\[P\left(\bigg|\sum_{i=1}^n v_i\epsilon_i\bigg|>z\right)\leq O(\exp(-c_\epsilon z^2)),\]
where $c_\epsilon$ is some positive constant depending on the distribution of $\epsilon$.
\end{prop}

According to Definition 1 in \cite{holp} and Proposition 5.16 in \cite{ver}, similar probability bounds exist for weighted sums of sub-exponential distributed variables. Thus, our main theorems hold for sub-exponential distributed random errors with slightly different probability bounds. For simplicity, we only consider sub-Gaussian distributed random errors in our proof, including normal distributed, Bernoulli distributed and other bounded random errors. 

Finally, by Proposition 5.10 and 5.16 in \cite{ver}, assumption (A4) holds for the normally distributed response and the bounded response. From the above discussion, we notice that assumptions (A1)-(A4) can be achieved simultaneously in at least two common scenarios, where the predictors and random error follow the normal distribution with a covariance matrix satisfying condition \eqref{covm} or all of them as well as the response are bounded with the same covariance matrix. It is also noteworthy that our proof does not need the marginal correlation assumption that the sure screening property of SIS relies on, which corroborates with the numerical results in \cite{sis} and \cite{vanisis}.

\subsection{Main theorems}
In the following theorems, we formally describe the sure screening properties of three types of iterative screening methods corresponding to the three selection criteria, where each type includes three algorithms employing different screening criterion but the same selection criterion. Then, the sure screening properties of ISIS, Van-ISIS and Van-ISIS-R can be obtained directly from these results.

We begin with algorithms that apply the Selection Criterion 1, where no penalized variable selection technique is involved in each iteration.
\begin{thm}\label{thm:sel1}
Let $\mathcal{S}_\kappa$ be the model obtained at the $\kappa$-th step of the iterative algorithm applying the Selection Criterion 1, satisfying that $\kappa\geq c_\kappa n^{\xi_y+2\xi_\beta}$ with $c_\kappa=8c_y\tau^3_{\text{max}}/(c_\beta^2\tau_{\text{min}}^4)$ and $\sum_{k=1}^\kappa a_k\leq c_sn^{\xi_s}$ for some positive constants $c_s$ and $\xi_s$ with $\xi_p+3\xi_s<1$. Then under assumptions (A1), (A2), (A3) and (A4), we have
\[P\left(\mathcal{T}\subset \mathcal{S}_\kappa\right)\geq 1-O\left(\exp\left(-c n^{\xi}\right)\right),\]
where $c$ is some positive constant and $\xi=\min\{{\xi_Y}, \xi_p+3\xi_s^*-2\xi_{\beta}, {1-2\xi_s^*}\}$ with $\xi_s^*=\xi_t\vee \xi_s$.
\end{thm}
\begin{rmk}
The non-penalized versions of ISIS, Van-ISIS and Van-ISIS-R can be regarded as algorithms that apply the Selection Criterion 1 and FR is equivalent to the non-penalized Van-ISIS with $a_k=1$ for $k\geq 1$. Therefore, their sure screening properties can be achieved directly from Theorem \ref{thm:sel1}.
\end{rmk}

\begin{rmk} Under the assumption that $\bm{x}\sim N(\bm{0},\Sigma)$ with $\Sigma$ satisfying condition \eqref{covm},  \cite{fr} proved that FR could identify the true model within $O(n^{2\xi_t+4\xi_\beta})$ steps with an overwhelming probability.

Under the same assumption, we have
\[c_y n^{\xi_y}\geq\mathrm{var}(y)>\bm{\beta}^\top \Sigma\bm{\beta}\geq \lambda_{\min}(\Sigma)\cdot||\bm{\beta}||^2\geq 2\tau_{\min}c_{\beta}^2n^{-2\xi_\beta}\cdot t,\] 
which indicates that 
\[t\leq c_y n^{\xi_y+2\xi_\beta}/(2\tau_{\min}c_{\beta}^2).\]
Therefore, in this case one can take $\xi_t\geq \xi_y+2\xi_\beta$, and the minimum required number of iteration for the sure screening of FR indicated by Theorem \ref{thm:sel1} ($O(n^{\xi_y+2\xi_\beta})$) is less than the square root of that ($O(n^{2\xi_t+4\xi_\beta})$) in \cite{fr}.
\end{rmk}

Next, we investigate the sure screening properties of algorithms applying the Selection Criterion 2 and Selection Criterion 3, where LASSO or SCAD is employed in the determination of selected models.

\begin{thm}\label{thm:sel2}
Let $\mathcal{S}_\kappa$ be the model obtained at the $\kappa$-th step of the iterative algorithm applying the Selection Criterion 2 using LASSO or SCAD with a tuning parameter $\lambda_n\leq \tau_{\min}^{2}c_\beta n^{-\xi_\beta}/(4\tau_{\max})$. Then if $\kappa\geq 2c_\kappa n^{\xi_y+2\xi_\beta}$ with $c_\kappa=8c_y\tau^3_{\text{max}}/(c_\beta^2\tau_{\text{min}}^4)$ and $\sum_{k=1}^\kappa a_k\leq c_sn^{\xi_s}$ for some positive constants $c_s$ and $\xi_s$ with $\xi_p+3\xi_s<1$, under assumptions (A1), (A2), (A3) and (A4), we have
\[P\left(\mathcal{T}\subset \mathcal{S}_\kappa\right)\geq 1-O\left(\exp\left(-c n^{\xi}\right)\right),\]
where $c$ is some positive constant and $\xi=\min\{{\xi_Y}, \xi_p+3\xi_s^*-2\xi_{\beta}, {1-2\xi_s^*}\}$ with $\xi_s^*=\xi_t\vee \xi_s$.
\end{thm}

\begin{rmk}
ISIS is equivalent to the algorithm applying the Selection Criterion 2 with the Screening Criterion 1 and thus its sure screening property can be obtained from Theorem \ref{thm:sel2}.
\end{rmk}

For the iterative algorithms applying the Selection Criterion 3, predictors selected in previous steps can be screened out in the later iterations. Therefore, the statement $\mathcal{T}\subset\mathcal{S}_{k}$ does not necessarily imply that $\mathcal{T}\subset\mathcal{S}_{k+1}$. Then we consider a weakened sure screening property for this third type of algorithms in the following theorem.
\begin{thm}\label{thm:sel3}
Let $\mathcal{S}_\kappa$ be the model obtained at the $\kappa$-th step of the
iterative algorithm applying the Selection Criterion 3 with LASSO or SCAD, satisfying that $\kappa\geq 2c_\kappa n^{\xi_y+2\xi_\beta}$ with $c_\kappa=8c_y\tau^3_{\text{max}}/(c_\beta^2\tau_{\text{min}}^4)$, $\sum_{k=1}^\kappa a_k\leq c_sn^{\xi_s}$ for some positive constants $c_s$ and $\xi_s$ with $\xi_p+3\xi_s<1$ and $\lambda_n\leq{\tau_{\min}^{4}c_\beta^2}n^{-2\xi_\beta-(\xi_s^*+\xi_y)/2}/({8\sqrt{c^*}\tau_{\max}^3})$ with $c^*=4c_y(c_t\vee c_s) /\tau_{\min}$. Then under assumptions (A1), (A2), (A3) and (A4), we have
\[P\left(\mathcal{T}\subset\mathcal{S}_k\text{ for some }1\leq k\leq \kappa\right)\geq 1-O\left(\exp\left(-c n^{\xi}\right)\right),\]
where $c$ is some positive constant and $\xi=\min\{{\xi_Y}, \xi_p+3\xi_s^*-2\xi_{\beta}, {1-2\xi_s^*}\}$ with $\xi_s^*=\xi_t\vee \xi_s$.
\end{thm}

\begin{rmk}
Theorem \ref{thm:sel3} indicates that both Van-ISIS and Van-ISIS-R could identify the true model at least once in the first $\kappa$ iterations. Notice that when applying the Selection Criterion 3, if we obtain $\mathcal{S}_k=\mathcal{S}_{k+1}$ for some $k$, then we have  $\mathcal{S}_k=\mathcal{S}_{k+j}$ for any $j>0$. This is the reason to terminate the Van-ISIS and Van-ISIS-R algorithms in real data analysis when $\mathcal{S}_k=\mathcal{S}_{k+1}$ is achieved. Intuitively, if one could find a consistent variable selection method with $P({\mathcal{M}}=\mathcal{T})$ converging to $1$ exponentially fast, where ${\mathcal{M}}$ denotes the submodel selected from any model that covers the true model, then one can show that
$\mathcal{T}\subset \mathcal{S}_{\kappa}$ with an overwhelming probability in the third type of algorithms.
\end{rmk}

\section{Technical details} In this section, we introduce several necessary results for the proof of our main theorems, and their proofs can be referred to the appendix. Based on these results, we then prove the sure screening properties of the three types of iterative screening methods.
\subsection{Preliminary results}
Initially, we establish a lower bound for the reduction of RSS in terms of the rescaled marginal correlations between predictors and the residual. 
\begin{prop}\label{prop:rsslb}
For any disjoint index sets $\mathcal S$ and $\mathcal A$, denote $X_{\mathcal S\cup \mathcal A}=[X_{\mathcal S}, X_{\mathcal A}]$. If $X_{\mathcal{S}\cup \mathcal A}$ is of full column rank, then we have
\[||{M}_{\mathcal S}{Y}||^2-||{M}_{\mathcal S\cup \mathcal A}{Y}||^2 \geq \sum_{i\in \mathcal A}({X}^\top _i{M}_{\mathcal S}{Y})^2/\lambda_{\text{max}}({X}^\top _{\mathcal A}{M}_{\mathcal S}{X}_{\mathcal A}).\]
\end{prop}
\begin{cor}\label{cor:2to1}
For any index set $\mathcal S$ and any index $i\not\in \mathcal S$, denote $X_{\mathcal S\cup \{i\}}=[X_{\mathcal S}, X_i]$. If $X_{\mathcal S\cup \{i\}}$ is of full column rank, we have
\[||{M}_{\mathcal S}{Y}||^2-||{M}_{\mathcal S\cup \{i\}}{Y}||^2 = ({X}^\top _i{M}_{\mathcal S}{Y})^2/||X_i^\top {M}_{\mathcal S}||^2.\]
\end{cor}

Corollary \ref{cor:2to1} indicates that choosing the smallest $||M_{\mathcal{S}_k\cup\{i\}}Y||^2$ is equivalent to finding the largest $({X}^\top _i{M}_{\mathcal{S}_k}{Y})^2/||X_i^\top {M}_{\mathcal S_k}||^2$ in the Screening Criterion 2. Moreover, we could also establish the following relationship between $\hat\beta_i$ evaluated in the Screening Criterion 3 and the marginal correlation.
\begin{prop}\label{prop:3to1}
For any index set $\mathcal S$ and any index $i\not\in \mathcal S$, we compute $\hat\beta_i$ as
\[(\hat{\bm{\beta}}_{\mathcal S},\hat\beta_i)=\argmin_{{\bm{\beta}}_{\mathcal S}\in\mathbb{R}^{|\mathcal S|},\beta\in\mathbb{R}}\big|\big|Y-X_{\mathcal S}^\top \bm{\beta}_{\mathcal S}-X_i\beta\big|\big|^2.\]
If $X_{\mathcal S\cup \{i\}}=[X_{\mathcal S},X_i]$ is of full column rank, we have
\[\hat\beta_i^2=({X}^\top _i{M}_{\mathcal S}{Y})^2/||X_i^\top {M}_{\mathcal S}||^4.\]
\end{prop}

Proposition \ref{prop:3to1} further implies that choosing the largest $|\hat\beta_i|$ in the Screening Criterion 3 is equivalent to selecting the largest $({X}^\top _i{M}_{\mathcal S_k}{Y})^2/||X_i^\top {M}_{\mathcal S_k}||^4$. From Corollary \ref{cor:2to1} and Proposition \ref{prop:3to1}, we see that the Screening Criteria 1-3 can be applied through evaluating scaled terms of $({X}^\top _i{M}_{\mathcal S_k}{Y})^2$. In the following proposition, we establish a lower bound of the maximum value of $({X}^\top _i{M}_{\mathcal S_k}{Y})^2$ for unidentified relevant predictors.

\begin{prop}\label{prop:lb1}
For any index set $\mathcal S$ and the true model $\mathcal{T}$, denote $\mathcal{T}_{\mathcal S}^c=\mathcal{T}-\mathcal S$ and $\mathcal S^*= \mathcal S\cup \mathcal T$. Then, if $\mathcal {T}_{\mathcal S}^c$ is non-empty and ${X}_{\mathcal S^*}=[{X}_{\mathcal S}, {X}_{{\mathcal T}_{\mathcal S}^c}]$ is of full column rank, we have
\[\max\limits_{i\in {\mathcal T}_{\mathcal S}^c}({X}^\top _i{M}_{\mathcal S}{Y})^2\geq \frac{\beta^2_{\text{min}}}{2}\cdot\lambda^2_{\text{min}}({X}^\top _{\mathcal S^*}{X}_{\mathcal S^*})
-\max\limits_{i\in {\mathcal T}_{\mathcal S}^c}({X}^\top _{i}{X}_{i})\cdot\max\limits_{i\in {\mathcal T}_{\mathcal S}^c}\left(\frac{{X}^\top _i{M}_{\mathcal S}{\bm\epsilon}}{||{X}^\top _i{M}_{\mathcal S}||}\right)^2.\]
\end{prop}
\subsection{Proof of main theorems}
\begin{proof}[Proof of Theorem \ref{thm:sel1}]
Recall that when applying the Selection Criterion 1, model $\mathcal{S}_{k+1}$ is determined as $\mathcal{S}_{k+1}=\mathcal{S}_{k}\cup\mathcal{A}_{k+1}$. Let $\kappa=c_\kappa n^{\xi_y+2\xi_\beta}$ with $c_\kappa=8c_y\tau^3_{\text{max}}/(c_\beta^2\tau_{\text{min}}^4)$. Without loss of generality, we assume that $\kappa$ is an integer. Denoting $\mathcal{S}_{\kappa}^*=\mathcal{S}_{\kappa}\cup\mathcal{T}$, under the assumption that $t\leq c_tn^{\xi_t}$ and $\sum_{k=1}^\kappa a_k\leq c_sn^{\xi_s}$, we have
\[|\mathcal{S}_{\kappa}^*|\leq t + \sum_{k=1}^\kappa a_k\leq s_{\kappa}^* = 2C_s^*n^{\xi_s^*},\]
where $C_s^*=c_t\vee c_s$ and $\xi_s^*=\xi_t\vee \xi_s$. Moreover, we notice that $\sum_{k=1}^\kappa a_k\geq \kappa$, indicating that $\xi_s\geq \xi_y+2\xi_\beta>\xi_\beta$.

Let $\mathcal{E}$ denote the event that $\mathcal{T}\not\subset\mathcal{S}_{\kappa}$. To achieve the sure screening property, it is sufficient to prove that $\lim_{n\rightarrow \infty}P(\mathcal{E})=0$. Moreover, define the event $\mathcal{E}_\lambda$ as
\[\mathcal{E}_\lambda = \left\{\tau_{\text{min}}\leq \min\limits_{|\mathcal{S}|\leq s_{\kappa}^*}\lambda_{\text{min}}\{\hat{{\Sigma}}_{(\mathcal{S})}\}\leq \max\limits_{|\mathcal{S}|\leq s_{\kappa}^*}\lambda_{\text{max}}\{\hat{{\Sigma}}_{(\mathcal{S})}\}\leq\tau_{\text{max}}\right\}.\]
Thus, according to assumption (A2) with $\xi_p+3\xi_s^*<1$, there exists some positive constant $c_s^*$, such that
\begin{equation}\label{eq:elambda}
P\left(\mathcal{E}_\lambda^c\right)\leq O\left(\exp\left(-c_s^* n^{1-2\xi_s^*}\right)\right),
\end{equation}
where $\mathcal{E}_\lambda^c$ denotes the complement of event $\mathcal{E}_\lambda$.

Finally, for $\xi_\epsilon=\xi_p+3\xi_s^*-2\xi_{\beta}$, we consider the event 
\[\mathcal{E}_\epsilon=\left\{\max_{\substack{i\in \mathcal{T}\\ \mathcal S\not\ni i,\, |\mathcal S|< s_{\kappa}^*}}\left(\frac{{X}^\top _i{M}_{\mathcal S}\bm{\epsilon}}{||{X}^\top _i{M}_{\mathcal S}||}\right)^2< n^{\xi_\epsilon}\right\}.\]
Notice that when the event $\mathcal{E}_\lambda$ holds, for any $i\in \mathcal{T}$ and any $\mathcal S\not\ni i$ with $|\mathcal S|<s_{\kappa}^*$, $X_{\mathcal S}$ is of full column rank and  
\begin{align}
||{X}^\top _i{M}_{\mathcal S}||^2&=||[I_n-X_{\mathcal S}(X_{\mathcal S}^\top X_{\mathcal S})^{-1}X_{\mathcal S}^\top ]X_i||^2\nonumber\\
& = ||[X_i,X_{\mathcal S}][1,X_i^\top X_{\mathcal S}(X_{\mathcal S}^\top X_{\mathcal S})^{-1}]^\top ||^2\nonumber\\
& \geq\lambda_{\min}(X_{\mathcal S\cup\{i\}}^\top X_{\mathcal S\cup\{i\}})\cdot(1+||X_i^\top X_{\mathcal S}(X_{\mathcal S}^\top X_{\mathcal S})^{-1}||^2)\nonumber\\
& \geq n\tau_{\min}>0,\label{eq:xmlb}
\end{align}
where $X_{\mathcal S\cup\{i\}}=[X_{\mathcal S},X_i]$. Therefore, ${X}^\top _i{M}_{\mathcal S}/||{X}^\top _i{M}_{\mathcal S}||$ is well defined and has the unit $L_2$ norm. Consequently, by Proposition \ref{prop:tail} and assumption (A1), we have
\begin{align*}
P\left(\mathcal{E}_\epsilon^c\cap\mathcal{E}_\lambda\right)&\leq\sum_{i\in \mathcal{T}}\sum_{{\mathcal S}\not\ni i,\, |\mathcal S|< s_{\kappa}^*}P\left(\left\{\left(\frac{{X}^\top _i{M}_{\mathcal S}\bm{\epsilon}}{||{X}^\top _i{M}_{\mathcal S}||}\right)^2\geq n^{\xi_\epsilon}\right\}\bigcap\mathcal{E}_\lambda\right)\\
&\leq t\cdot p^{s_{\kappa}^*}\cdot O\left(\exp\left(-c_\epsilon n^{\xi_\epsilon}\right)\right)\\
&\leq O\left(n^{\xi_t}\cdot\exp\left(2C_s^*c_pn^{\xi_p+\xi_s^*}-c_\epsilon n^{\xi_\epsilon}\right)\right)\\
&\leq O\left(\exp\left(-c_\epsilon n^{\xi_\epsilon}/2\right)\right),
\end{align*}
where $\xi_p+\xi_s^*<\xi_\epsilon=\xi_p+3\xi_s^*-2\xi_{\beta}$ since $\xi_s^*>\xi_{\beta}$. Therefore, from the inequality \eqref{eq:elambda}, we obtain that
\begin{align}
P\left(\mathcal{E}_\epsilon^c\right)&\leq P\left(\mathcal{E}_\epsilon^c\cap\mathcal{E}_\lambda\right)+ P\left(\mathcal{E}_\lambda^c\right)\nonumber\\
&\leq O\left(\exp\left(-c_\epsilon n^{\xi_\epsilon}/2\right)\right)+O\left(\exp\left(-c_s^* n^{1-2\xi_s^*}\right)\right).\label{eq:eepsilon}
\end{align}

Denoting $\mathcal S_{0}=\emptyset$ and $M_{\mathcal S_{0}}=I_n$, for $k\geq 1$, we define
\[\rho_k=\min_{\bm{\beta}\in\mathbb{R}^{{\scriptstyle s}_{\scalebox{.7}{$\scriptscriptstyle k-1$}}}}||Y-X_{\mathcal{S}_{k-1}}\bm{\beta}||^2= ||M_{\mathcal{S}_{k-1}}Y||^2.\]

In the next, we will prove the uniform lower bound for $\rho_k-\rho_{k+1}$ for $1\leq k\leq \kappa$ under events $\mathcal{E}$, $\mathcal{E}_\lambda$ and $\mathcal{E}_\epsilon$. Notice that event $\mathcal{E}$ implies that $\mathcal{T}\not\subset\mathcal{S}_k$ for any $1\leq k\leq \kappa$. Thus, we have $\mathcal{T}_{{k-1}}^c=\mathcal{T}-\mathcal{S}_{k-1}\neq \emptyset$ and define
\[i_{k}^*=\argmax_{i\in \mathcal{T}_{{k-1}}^c}(X_{i}^\top M_{\mathcal{S}_{k-1}}Y)^2.\]
Denoting $\mathcal{S}_{k-1}^*=\mathcal{S}_{k-1}\cup\mathcal{T}$, from Proposition \ref{prop:lb1}, we obtain that
\begin{align*}
&~~~~(X_{i_{k}^*}^\top M_{\mathcal{S}_{k-1}}Y)^2
=\max_{i\in \mathcal{T}_{{k-1}}^c}(X_{i}^\top M_{\mathcal{S}_{k-1}}Y)^2\\
&\geq \frac{\beta^2_{\text{min}}}{2}\cdot\lambda^2_{\text{min}}({X}^\top _{\mathcal{S}_{k-1}^*}{X}_{\mathcal{S}_{k-1}^*})
-\max\limits_{i\in \mathcal{T}_{\mathcal{S}_{k-1}}^c}({X}^\top _{i}{X}_{i})\cdot\max\limits_{i\in \mathcal{T}_{\mathcal{S}_{k-1}}^c}\left(\frac{{X}^\top _iM_{\mathcal{S}_{k-1}}\bm{\epsilon}}{||{X}^\top _iM_{\mathcal{S}_{k-1}}||}\right)^2,
\end{align*}
which together with events $\mathcal{E}_\lambda$ and $\mathcal{E}_\epsilon$, imply that
\begin{align}
(X_{i_{k}^*}^\top M_{\mathcal{S}_{k-1}}Y)^2 &\geq \frac{c_\beta^2\tau_{\min}^2}{2}n^{2-2\xi_\beta}-\tau_{\max}\cdot n^{1+\xi_\epsilon}\nonumber\\
&= \frac{c_\beta^2\tau_{\min}^2}{2}n^{2-2\xi_\beta}\left[1-\frac{2\tau_{\max}}{c_\beta^2\tau_{\min}^2}n^{\xi_p+3\xi_s^*-1}\right]\nonumber\\
&\geq \frac{c_\beta^2\tau_{\min}^2}{4}n^{2-2\xi_\beta},\label{eq:marglb}
\end{align}
where the last inequality comes from the fact that $\xi_p+3\xi_s^*<1$. Then, the lower bound of $\rho_k-\rho_{k+1}$ can be evaluated separately in the $i_{k}^*\in\mathcal{A}_k$ and $i_{k}^*\not\in\mathcal{A}_k$ cases.

\textbf{Case 1:} If $i_{k}^*\in\mathcal{A}_k$, by Proposition \ref{prop:rsslb}, we have
\begin{align}
\rho_k-\rho_{k+1}&=||M_{\mathcal{S}_{k-1}}Y||^2-||M_{\mathcal{S}_{k}}Y||^2\nonumber\\
&\geq \sum_{i\in \mathcal{A}_{k}}({X}^\top _i{M}_{\mathcal{S}_{k-1}}{Y})^2/\lambda_{\text{max}}({X}^\top _{\mathcal{A}_k}M_{\mathcal{S}_{k-1}}{X}_{\mathcal{A}_k})\nonumber\\
&\geq (X_{i_{k}^*}^\top M_{\mathcal{S}_{k-1}}Y)^2/\lambda_{\text{max}}({X}^\top _{\mathcal{A}_k}{X}_{\mathcal{A}_k}).\label{eq:rholb}
\end{align}
Then by event $\mathcal{E}_\lambda$ and inequality \eqref{eq:marglb}, inequality \eqref{eq:rholb} is followed by
\begin{equation}\label{eq:lb1.1}
\rho_k-\rho_{k+1}\geq\frac{c_\beta^2\tau_{\min}^2}{4}n^{2-2\xi_\beta}/n\tau_{\max}=\frac{c_\beta^2\tau_{\min}^2}{4\tau_{\max}}n^{1-2\xi_\beta}.
\end{equation}

\textbf{Case 2:} If $i_{k}^*\not\in\mathcal{A}_k$, let $i_{k}^{**}$ denote an arbitrary element in $\mathcal{A}_k$. Under events $\mathcal{E}$, $\mathcal{E}_\lambda$ and $\mathcal{E}_\epsilon$, the lower bound of $\rho_k-\rho_{k+1}$ in algorithms applying the three different screening criteria can be evaluated as follows.

\textbf{Screening Criterion 1:} When applying the Screening Criterion 1, the set of candidate predictors is determined as
\[\mathcal{A}_{k}=\{i \in\mathcal{S}_{k-1}^c:\,\,|X_i^\top M_{\mathcal{S}_{k-1}}Y| \text{ is among the first } a_{k} \text{ largest of all}\}.\]
Consequently, with $i_{k}^*\not\in\mathcal{A}_k$ and $i_{k}^{**}\in\mathcal{A}_k$, we have
\[(X_{i_{k}^{**}}^\top M_{\mathcal{S}_{k-1}}Y)^2\geq (X_{i_{k}^*}^\top M_{\mathcal{S}_{k-1}}Y)^2.\]
Therefore, under event $\mathcal{E}_\lambda$, inequalities \eqref{eq:marglb} and \eqref{eq:rholb} imply that
\begin{align}
\rho_k-\rho_{k+1}&\geq(X_{i_{k}^{**}}^\top M_{\mathcal{S}_{k-1}}Y)^2/\lambda_{\text{max}}({X}^\top _{\mathcal{A}_k}{X}_{\mathcal{A}_k})\nonumber\\
&\geq(X_{i_{k}^*}^\top M_{\mathcal{S}_{k-1}}Y)^2/\lambda_{\text{max}}({X}^\top _{\mathcal{A}_k}{X}_{\mathcal{A}_k})\nonumber\\
&\geq \frac{c_\beta^2\tau_{\min}^2}{4\tau_{\max}}n^{1-2\xi_\beta}.\label{eq:lb1.2.1}
\end{align}

\textbf{Screening Criterion 2:} When applying the Screening Criterion 2, $\mathcal{A}_{k}$ is selected as
\[\mathcal{A}_{k}=\{i \in\mathcal{S}_{k-1}^c:\,\,||M_{\mathcal{S}_{k-1}\cup\{i\}}Y||^2 \text{ is among the first } a_{k} \text{ smallest of all}\}.\]
According to Corollary \ref{cor:2to1}, we have
\[\frac{(X_{i_{k}^{**}}^\top M_{\mathcal{S}_{k-1}}Y)^2}{||X_{i_{k}^{**}}^\top M_{\mathcal{S}_{k-1}}||^2}\geq \frac{(X_{i_{k}^{*}}^\top M_{\mathcal{S}_{k-1}}Y)^2}{||X_{i_{k}^{*}}^\top M_{\mathcal{S}_{k-1}}||^2}.\]
From event $\mathcal{E}_\lambda$ and inequality \eqref{eq:xmlb}, for any $i \in\mathcal{S}_{k-1}^c$, we have
\begin{equation}\label{xmulb}
n\tau_{\min}\leq ||X_{i}^\top M_{\mathcal{S}_{k-1}}||^2\leq X_{i}^\top X_{i}\leq n\tau_{\max}.
\end{equation}
Then we obtain that
\begin{align}
(X_{i_{k}^{**}}^\top M_{\mathcal{S}_{k-1}}Y)^2&\geq ||X_{i_{k}^{**}}^\top M_{\mathcal{S}_{k-1}}||^2\cdot\frac{(X_{i_{k}^{*}}^\top M_{\mathcal{S}_{k-1}}Y)^2}{||X_{i_{k}^{*}}^\top M_{\mathcal{S}_{k-1}}||^2}\nonumber\\
&\geq \frac{\tau_{\min}}{\tau_{\max}}\cdot(X_{i_{k}^{*}}^\top M_{\mathcal{S}_{k-1}}Y)^2.\label{eq:xmlb2}
\end{align}
Thus, from inequalities \eqref{eq:marglb}, \eqref{eq:rholb} and \eqref{eq:xmlb2}, we achieve that
\begin{align}
\rho_k-\rho_{k+1}&\geq(X_{i_{k}^{**}}^\top M_{\mathcal{S}_{k-1}}Y)^2/\lambda_{\text{max}}({X}^\top _{\mathcal{A}_k}{X}_{\mathcal{A}_k})\nonumber\\
&\geq\frac{\tau_{\min}}{\tau_{\max}}\cdot \frac{(X_{i_{k}^{*}}^\top M_{\mathcal{S}_{k-1}}Y)^2}{\lambda_{\text{max}}({X}^\top _{\mathcal{A}_k}{X}_{\mathcal{A}_k})}\geq \frac{c_\beta^2\tau_{\min}^3}{4\tau_{\max}^2}n^{1-2\xi_\beta}.\label{eq:lb1.2.2}
\end{align}

\textbf{Screening Criterion 3:} When applying the Screening Criterion 3, $\mathcal{A}_{k}$ is chosen as
\[\mathcal{A}_{k}=\{i \in\mathcal{S}_{k-1}^c:\,\,|\hat\beta_i| \text{ is among the first } a_{k} \text{ largest of all}\}.\]
According to Proposition \ref{prop:3to1}, we have
\[\frac{(X_{i_{k}^{**}}^\top M_{\mathcal{S}_{k-1}}Y)^2}{||X_{i_{k}^{**}}^\top M_{\mathcal{S}_{k-1}}||^4}\geq \frac{(X_{i_{k}^{*}}^\top M_{\mathcal{S}_{k-1}}Y)^2}{||X_{i_{k}^{*}}^\top M_{\mathcal{S}_{k-1}}||^4}.\]
Then, from inequality \eqref{xmulb}, we achieve that
\begin{align}
(X_{i_{k}^{**}}^\top M_{\mathcal{S}_{k-1}}Y)^2&\geq ||X_{i_{k}^{**}}^\top M_{\mathcal{S}_{k-1}}||^4\cdot\frac{(X_{i_{k}^{*}}^\top M_{\mathcal{S}_{k-1}}Y)^2}{||X_{i_{k}^{*}}^\top M_{\mathcal{S}_{k-1}}||^4}\nonumber\\
&\geq \frac{\tau_{\min}^2}{\tau_{\max}^2}\cdot(X_{i_{k}^{*}}^\top M_{\mathcal{S}_{k-1}}Y)^2.\label{eq:xmlb3}
\end{align}
Similarly, from inequalities \eqref{eq:marglb}, \eqref{eq:rholb} and \eqref{eq:xmlb3}, we also obtain that
\begin{align}
\rho_k-\rho_{k+1}&\geq(X_{i_{k}^{**}}^\top M_{\mathcal{S}_{k-1}}Y)^2/\lambda_{\text{max}}({X}^\top _{\mathcal{A}_k}{X}_{\mathcal{A}_k})\nonumber\\
&\geq\frac{\tau_{\min}^2}{\tau_{\max}^2}\cdot \frac{(X_{i_{k}^{*}}^\top M_{\mathcal{S}_{k-1}}Y)^2}{\lambda_{\text{max}}({X}^\top _{\mathcal{A}_k}{X}_{\mathcal{A}_k})}\geq \frac{c_\beta^2\tau_{\min}^4}{4\tau_{\max}^3}n^{1-2\xi_\beta}.\label{eq:lb1.2.3}
\end{align}

Combining the results in inequalities \eqref{eq:lb1.1}, \eqref{eq:lb1.2.1}, \eqref{eq:lb1.2.2} and \eqref{eq:lb1.2.3}, with the fact that $\tau_{\min}<\tau_{\max}$, we obtain the following uniform lower bound for $\rho_k-\rho_{k+1}$ under events $\mathcal{E}$, $\mathcal{E}_\lambda$ and $\mathcal{E}_\epsilon$ in both $i_{k}^*\in\mathcal{A}_k$ and $i_{k}^*\not\in\mathcal{A}_k$ cases,
\begin{equation}\label{eq:lb1}
\rho_k-\rho_{k+1}\geq\frac{c_\beta^2\tau_{\min}^4}{4\tau_{\max}^3}n^{1-2\xi_\beta}.
\end{equation}
Consequently, events $\mathcal{E}$, $\mathcal{E}_\lambda$ and $\mathcal{E}_\epsilon$ together imply that
\[
\rho_1\geq\rho_1-\rho_{\kappa+1}=\sum_{k=1}^\kappa (\rho_k-\rho_{k+1})\geq \kappa\cdot \frac{c_\beta^2\tau_{\min}^4}{4\tau_{\max}^3}n^{1-2\xi_\beta}= 2c_y n^{1+\xi_y}.
\]
Notice that $\rho_1=||Y||^2$ and $\sigma_y^2\leq c_y n^{\xi_y}$. Then, under assumption (A4), we have
\begin{align*}
P\left(\rho_1\geq 2c_y n^{1+\xi_y}\right)&=P\left(\frac{1}{n}||Y||^2\geq 2c_y n^{\xi_y}\right)\\
&\leq P\left(\frac{1}{n}||Y||^2\geq 2\sigma_y^2\right)\leq O\left(\exp\left(-c_Yn^{\xi_Y}\right)\right),
\end{align*}
which indicates that
\[P\left(\mathcal{E}\cap\mathcal{E}_\lambda\cap\mathcal{E}_\epsilon\right)\leq O\left(\exp\left(-c_Yn^{\xi_Y}\right)\right).\]
Finally, from probability bounds \eqref{eq:elambda} and \eqref{eq:eepsilon}, we obtain that
\begin{align*}
&~~~~P\left(\mathcal{T}\not\subset\mathcal{S}_\kappa\right)=P\left(\mathcal{E}\right)\leq P\left(\mathcal{E}\cap\mathcal{E}_\lambda\cap\mathcal{E}_\epsilon\right)+P\left(\mathcal{E}_\lambda^c\right)+P\left(\mathcal{E}_\epsilon^c\right)\\
&\leq O\left(\exp\left(-c_Yn^{\xi_Y}\right)\right)+O\left(\exp\left(-c_\epsilon n^{\xi_\epsilon}/2\right)\right)+O\left(\exp\left(-c_s^* n^{1-2\xi_s^*}\right)\right)\\
&\leq O\left(\exp\left(-c n^{\xi}\right)\right),
\end{align*}
where $c$ is some positive constant and $\xi=\min\{{\xi_Y}, \xi_p+3\xi_s^*-2\xi_{\beta}, {1-2\xi_s^*}\}$ with $\xi_s^*=\xi_t\vee \xi_s$.
\end{proof}

\begin{proof}[Proof of Theorem \ref{thm:sel2}]
In the Selection Criterion 2, model $\mathcal{S}_{k+1}$ is determined as $\mathcal{S}_{k+1}=\mathcal{S}_{k}\cup\mathcal{B}_{k+1}$, where $\mathcal{B}_{k+1}$ is obtained by solving the PLS problem
\[\min_{\bm{\beta}\in\mathbb{R}^{{\scriptstyle a}_{\scalebox{.7}{$\scriptscriptstyle k+1$}}}}\left[\big|\big|M_{\mathcal{S}_k}Y-X_{\mathcal{A}_{k+1}}\bm{\beta}\big|\big|^2+n\sum_{j=1}^{a_{k+1}}p_{\lambda_n}(|\beta_j|)\right].\]

In the theorem, we only consider LASSO and SCAD for simplicity. According to the definitions of their penalty functions in equations \eqref{p.lasso} and \eqref{p.scad}, we obtain that $p_{\lambda_n}(|\theta|)\leq \lambda_n|\theta|$ holds for both methods.

By setting $\kappa=2c_\kappa n^{\xi_y+2\xi_\beta}$ with $c_\kappa=8c_y\tau^3_{\text{max}}/(c_\beta^2\tau_{\text{min}}^4)$, we consider the following three events same as in the proof of Theorem \ref{thm:sel1}.
\begin{enumerate}[align=left, leftmargin=*]
\item The objective event is defined as $\mathcal{E}=\{\mathcal{T}\not\subset\mathcal{S}_\kappa\}$.
\item The event concerning extreme eigenvalues of Gram matrices is defined as 
\[\mathcal{E}_\lambda = \left\{\tau_{\text{min}}\leq \min\limits_{|\mathcal{S}|\leq s_{\kappa}^*}\lambda_{\text{min}}\{\hat{{\Sigma}}_{(\mathcal{S})}\}\leq \max\limits_{|\mathcal{S}|\leq s_{\kappa}^*}\lambda_{\text{max}}\{\hat{{\Sigma}}_{(\mathcal{S})}\}\leq\tau_{\text{max}}\right\},\]
where $s_{\kappa}^*=2C_s^*n^{\xi_s^*}$ with $C_s^*=c_t\vee c_s$ and $\xi_s^*=\xi_t\vee \xi_s$.
\item The event concerning weighted sums of elements in $\bm{\epsilon}$ is given by
\[\mathcal{E}_\epsilon=\left\{\max_{\substack{i\in \mathcal{T}\\ \mathcal{S}\not\ni i,\, |\mathcal {S}|<s_{\kappa}^*}}\left(\frac{{X}^\top _i{M}_{S}\bm{\epsilon}}{||{X}^\top _i{M}_{S}||}\right)^2< n^{\xi_\epsilon}\right\},\]
where $\xi_\epsilon=\xi_p+3\xi_s^*-2\xi_{\beta}$.
\end{enumerate}

For events $\mathcal{E}_\lambda^c$ and $\mathcal{E}_\epsilon^c$, we have the probability bounds \eqref{eq:elambda} and \eqref{eq:eepsilon}, respectively. Denoting $\mathcal{S}_{0}=\emptyset$ and $M_{\mathcal{S}_{0}}=I_n$, for $k\geq 1$, we define
\begin{align*}
\rho_k&=\min_{\bm{\beta}\in\mathbb{R}^{{\scriptstyle a}_{\scalebox{.7}{$\scriptscriptstyle k$}}}}\left[||M_{\mathcal{S}_{k-1}}Y-X_{\mathcal{A}_{k}}\bm{\beta}||^2 + n\sum_{j=1}^{a_k}p_{\lambda_n}(|\beta_j|)\right]\\
&=\min_{\bm{\beta}\in\mathbb{R}^{{\scriptstyle b}_{\scalebox{.7}{$\scriptscriptstyle k$}}}}\left[||M_{\mathcal{S}_{k-1}}Y-X_{\mathcal{B}_{k}}\bm{\beta}||^2 + n\sum_{j=1}^{b_k}p_{\lambda_n}(|\beta_j|)\right].
\end{align*}
From the definition of $\rho_k$, we obtain that
\begin{equation}\label{eq:thm2.1}
\rho_k\geq\min_{\bm{\beta}\in\mathbb{R}^{{\scriptstyle b}_{\scalebox{.7}{$\scriptscriptstyle k$}}}}||M_{\mathcal{S}_{k-1}}Y-X_{\mathcal{B}_{k}}\bm{\beta}||^2\geq||M_{\mathcal{S}_{k}}Y||^2,
\end{equation}
and
\begin{equation}\label{eq:thm2.2}
\rho_k\leq\left.\left[||M_{\mathcal{S}_{k-1}}Y-X_{\mathcal{B}_{k}}\bm{\beta}||^2 + n\sum_{j=1}^{b_k}p_{\lambda_n}(|\beta_j|)\right]\right|_{\bm{\beta}=0}\leq ||M_{\mathcal{S}_{k-1}}Y||^2.
\end{equation}

We then evaluate the uniform lower bound of $\rho_k-\rho_{k+1}$ for $1\leq k\leq \kappa$ when events $\mathcal{E}$, $\mathcal{E}_\lambda$ and $\mathcal{E}_\epsilon$ hold. For $1\leq k\leq \kappa$, we define $\mathcal{T}_{{k}}^c=\mathcal{T}-\mathcal{S}_{k}$ and
\[i_{k+1}^*=\argmax_{i\in \mathcal{T}_{{k}}^c}(X_{i}^\top M_{\mathcal{S}_{k}}Y)^2.\]

\textbf{Case 1:} If $i_{k+1}^*\in\mathcal{A}_{k+1}$, we have
\begin{align*}
\rho_{k+1}&=\min_{\bm{\beta}\in\mathbb{R}^{{\scriptstyle a}_{\scalebox{.7}{$\scriptscriptstyle k+1$}}}}\left[||M_{\mathcal{S}_{k}}Y-X_{\mathcal{A}_{k+1}}\bm{\beta}||^2 + n\sum_{j=1}^{a_{k+1}}p_{\lambda_n}(|\beta_j|)\right]\\
&\leq||M_{\mathcal{S}_{k}}Y-X_{i_{k+1}^*}\hat{\beta}^*||^2+n\lambda_n|\hat{\beta}^*|.
\end{align*}
where $\hat{\beta}^*=(X_{i_{k+1}^*}^\top X_{i_{k+1}^*})^{-1}X_{i_{k+1}^*}^\top M_{\mathcal{S}_{k}}Y$ is the OLS estimate of the coefficient in the componentwise regression between $M_{\mathcal{S}_{k}}Y$ and $X_{i_{k+1}^*}$. Therefore, from inequality \eqref{eq:thm2.1}, we have
\begin{align}
\rho_k-\rho_{k+1}&\geq ||M_{\mathcal{S}_{k}}Y||^2-||M_{i_{k+1}^*}M_{\mathcal{S}_{k}}Y||^2-n\lambda_n|\hat{\beta}^*|\nonumber\\
&= ||H_{i_{k+1}^*}M_{\mathcal{S}_{k}}Y||^2-n\lambda_n\frac{|X_{i_{k+1}^*}^\top M_{\mathcal{S}_{k}}Y|}{X_{i_{k+1}^*}^\top X_{i_{k+1}^*}}\nonumber\\
&= \frac{(X_{i_{k+1}^*}^\top M_{\mathcal{S}_{k}}Y)^2}{X_{i_{k+1}^*}^\top X_{i_{k+1}^*}}-n\lambda_n\frac{|X_{i_{k+1}^*}^\top M_{\mathcal{S}_{k}}Y|}{X_{i_{k+1}^*}^\top X_{i_{k+1}^*}}.\label{eq:thm2.3}
\end{align}
From inequality \eqref{eq:marglb}, we know that
\[(X_{i_{k+1}^*}^\top M_{\mathcal{S}_{k}}Y)^2\geq \frac{c_\beta^2\tau_{\min}^2}{4}n^{2-2\xi_\beta}.\] 
Therefore, if we choose $\lambda_n\leq c_\beta\tau_{\min}n^{-\xi_\beta}/4$, inequality \eqref{eq:thm2.3} is followed by 
\begin{equation}\label{eq:lb2.1}
\rho_k-\rho_{k+1}\geq\frac{(X_{i_{k+1}^*}^\top M_{\mathcal{S}_{k}}Y)^2}{2X_{i_{k+1}^*}^\top X_{i_{k+1}^*}}\geq \frac{c_\beta^2\tau_{\min}^2}{8\tau_{\max}}n^{1-2\xi_\beta}.
\end{equation}

\textbf{Case 2:} If $i_{k+1}^*\not\in\mathcal{A}_{k+1}$, let $i_{k+1}^{**}$ denote an arbitrary element in $\mathcal{A}_{k+1}$ and compute
\[\hat{\beta}^{**}=(X_{i_{k+1}^{**}}^\top X_{i_{k+1}^{**}})^{-1}X_{i_{k+1}^{**}}^\top M_{\mathcal{S}_{k}}Y.\]
Then we have
\[\rho_{k+1}\leq||M_{\mathcal{S}_{k}}Y-X_{i_{k+1}^{**}}\hat{\beta}^{**}||^2+n\lambda_n|\hat{\beta}^{**}|,\]
and
\begin{equation}\label{eq:lb2.2.0}
\rho_k-\rho_{k+1}\geq\frac{(X_{i_{k+1}^{**}}^\top M_{\mathcal{S}_{k}}Y)^2}{X_{i_{k+1}^{**}}^\top X_{i_{k+1}^{**}}}-n\lambda_n\frac{|X_{i_{k+1}^{**}}^\top M_{\mathcal{S}_{k}}Y|}{X_{i_{k+1}^{**}}^\top X_{i_{k+1}^{**}}}.
\end{equation}

In events $\mathcal{E}$, $\mathcal{E}_\lambda$ and $\mathcal{E}_\epsilon$, we evaluate the lower bound of $\rho_k-\rho_{k+1}$ for the three different screening criteria.

\textbf{Screening Criterion 1:} When applying the Screening Criterion 1,  we have
\[(X_{i_{k+1}^{**}}^\top M_{\mathcal{S}_{k}}Y)^2\geq (X_{i_{k+1}^*}^\top M_{\mathcal{S}_{k}}Y)^2\geq\frac{c_\beta^2\tau_{\min}^2}{4}n^{2-2\xi_\beta}.\]
Thus, if we choose the same tuning parameter satisfying $\lambda_n\leq c_\beta\tau_{\min}n^{-\xi_\beta}/4$, inequality \eqref{eq:lb2.2.0} is followed by 
\begin{equation}\label{eq:lb2.2.1}
\rho_k-\rho_{k+1}\geq\frac{(X_{i_{k+1}^{**}}^\top M_{\mathcal{S}_{k}}Y)^2}{2X_{i_{k+1}^{**}}^\top X_{i_{k+1}^{**}}}\geq \frac{c_\beta^2\tau_{\min}^2}{8\tau_{\max}}n^{1-2\xi_\beta}.
\end{equation}

\textbf{Screening Criterion 2:} In the Screening Criterion 2, from inequality \eqref{eq:xmlb2}, we have
\[
(X_{i_{k+1}^{**}}^\top M_{\mathcal{S}_{k}}Y)^2\geq \frac{\tau_{\min}}{\tau_{\max}}\cdot(X_{i_{k+1}^{*}}^\top M_{\mathcal{S}_{k}}Y)^2\geq\frac{c_\beta^2\tau_{\min}^3}{4\tau_{\max}}n^{2-2\xi_\beta}.
\]
If we choose $\lambda_n\leq c_\beta\tau_{\min}^{3/2}n^{-\xi_\beta}/(4\tau_{\max}^{1/2})$, we obtain that
\begin{equation}\label{eq:lb2.2.2}
\rho_k-\rho_{k+1}\geq\frac{(X_{i_{k+1}^{**}}^\top M_{\mathcal{S}_{k}}Y)^2}{2X_{i_{k+1}^{**}}^\top X_{i_{k+1}^{**}}}\geq \frac{c_\beta^2\tau_{\min}^3}{8\tau_{\max}^2}n^{1-2\xi_\beta}.
\end{equation}

\textbf{Screening Criterion 3:} In the Screening Criterion 3, from inequality \eqref{eq:xmlb3}, we have
\[
(X_{i_{k+1}^{**}}^\top M_{\mathcal{S}_{k}}Y)^2\geq \frac{\tau_{\min}^2}{\tau_{\max}^2}\cdot(X_{i_{k+1}^{*}}^\top M_{\mathcal{S}_{k}}Y)^2\geq\frac{c_\beta^2\tau_{\min}^4}{4\tau_{\max}^2}n^{2-2\xi_\beta}.
\]
If we choose $\lambda_n\leq c_\beta\tau_{\min}^{2}n^{-\xi_\beta}/(4\tau_{\max})$, we achieve that
\begin{equation}\label{eq:lb2.2.3}
\rho_k-\rho_{k+1}\geq\frac{(X_{i_{k+1}^{**}}^\top M_{\mathcal{S}_{k}}Y)^2}{2X_{i_{k+1}^{**}}^\top X_{i_{k+1}^{**}}}\geq \frac{c_\beta^2\tau_{\min}^4}{8\tau_{\max}^3}n^{1-2\xi_\beta}.
\end{equation}

Combining the results in inequalities \eqref{eq:lb2.1}, \eqref{eq:lb2.2.1}, \eqref{eq:lb2.2.2} and \eqref{eq:lb2.2.3}, we obtain that, when $\mathcal{E}$, $\mathcal{E}_\lambda$ and $\mathcal{E}_\epsilon$ hold, if we choose a tuning parameter satisfies that $\lambda_n\leq c_\beta\tau_{\min}^{2}n^{-\xi_\beta}/(4\tau_{\max})$, then 
\[
\rho_k-\rho_{k+1}\geq \frac{c_\beta^2\tau_{\min}^4}{8\tau_{\max}^3}n^{1-2\xi_\beta}.
\]
Consequently, events $\mathcal{E}$, $\mathcal{E}_\lambda$ and $\mathcal{E}_\epsilon$ together imply that
\[
\rho_1\geq\rho_1-\rho_{\kappa+1}=\sum_{k=1}^\kappa (\rho_k-\rho_{k+1})\geq \kappa\cdot \frac{c_\beta^2\tau_{\min}^4}{8\tau_{\max}^3}n^{1-2\xi_\beta}\geq 2c_y n^{1+\xi_y}.
\]
From inequality \eqref{eq:thm2.2}, we know that $\rho_1\leq||Y||^2$. Then, under assumption (A4), we have
\[
P\left(\rho_1\geq 2c_y n^{1+\xi_y}\right)\leq P\left(\frac{1}{n}||Y||^2\geq 2\sigma_y^2\right)\leq O\left(\exp\left(-c_Yn^{\xi_Y}\right)\right),
\]
which indicates that
\[P\left(\mathcal{E}\cap\mathcal{E}_\lambda\cap\mathcal{E}_\epsilon\right)\leq O\left(\exp\left(-c_Yn^{\xi_Y}\right)\right).\]
Consequently, from probability bounds \eqref{eq:elambda} and \eqref{eq:eepsilon}, we obtain that
\begin{align*}
&~~~~P\left(\mathcal{T}\not\subset\mathcal{S}_\kappa\right)=P\left(\mathcal{E}\right)\leq P\left(\mathcal{E}\cap\mathcal{E}_\lambda\cap\mathcal{E}_\epsilon\right)+P\left(\mathcal{E}_\lambda^c\right)+P\left(\mathcal{E}_\epsilon^c\right)\\
&\leq O\left(\exp\left(-c_Yn^{\xi_Y}\right)\right)+O\left(\exp\left(-c_\epsilon n^{\xi_\epsilon}/2\right)\right)+O\left(\exp\left(-c_s^* n^{1-2\xi_s^*}\right)\right)\\
&\leq O\left(\exp\left(-c n^{\xi}\right)\right),
\end{align*}
where $c$ is some positive constant and $\xi=\min\{{\xi_Y}, \xi_p+3\xi_s^*-2\xi_{\beta}, {1-2\xi_s^*}\}$ with $\xi_s^*=\xi_t\vee \xi_s$.
\end{proof}

\begin{proof}[Proof of Theorem \ref{thm:sel3}]
In the Selection Criterion 3, model $\mathcal{S}_{k+1}$ is obtained by minimizing the penalized least squares
\[\min_{\bm{\beta}\in\mathbb{R}^{{\scriptstyle s}_{\scalebox{.7}{$\scriptscriptstyle k$}}+{\scriptstyle a}_{\scalebox{.7}{$\scriptscriptstyle k+1$}}}}\left[\big|\big|Y-X_{\mathcal{S}_k\cup\mathcal{A}_{k+1}}\bm{\beta}\big|\big|^2+n\sum_{j=1}^{s_k+a_{k+1}}p_{\lambda_n}(|\beta_j|)\right].\] 

Setting $\kappa=2c_\kappa n^{\xi_y+2\xi_\beta}$ with $c_\kappa=8c_y\tau^3_{\text{max}}/(c_\beta^2\tau_{\text{min}}^4)$, we will prove that $\mathcal{T}\subset\mathcal{S}_k$ for some $1\leq k\leq \kappa$ with an overwhelming probability. Similarly, we consider the following four events.
\begin{enumerate}[align=left, leftmargin=*]
\item The objective event is defined as $\mathcal{E}=\bigcap_{k=1}^\kappa \{\mathcal{T}\not\subset\mathcal{S}_k\}$.
\item The event concerning extreme eigenvalues of Gram matrices is defined as 
\[\mathcal{E}_\lambda = \left\{\tau_{\text{min}}\leq \min\limits_{|\mathcal{S}|\leq s_{\kappa}^*}\lambda_{\text{min}}\{\hat{{\Sigma}}_{(\mathcal{S})}\}\leq \max\limits_{|\mathcal{S}|\leq s_{\kappa}^*}\lambda_{\text{max}}\{\hat{{\Sigma}}_{(\mathcal{S})}\}\leq\tau_{\text{max}}\right\},\]
where $s_{\kappa}^*=2C_s^*n^{\xi_s^*}$ with $C_s^*=c_t\vee c_s$ and $\xi_s^*=\xi_t\vee \xi_s$.
\item The event concerning weighted sums of elements in $\bm{\epsilon}$ is given by
\[\mathcal{E}_\epsilon=\left\{\max_{\substack{i\in \mathcal{T}\\ \mathcal{S}\not\ni i,\, |\mathcal {S}|<s_{\kappa}^*}}\left(\frac{{X}^\top _i{M}_{S}\bm{\epsilon}}{||{X}^\top _i{M}_{S}||}\right)^2< n^{\xi_\epsilon}\right\},\]
where $\xi_\epsilon=\xi_p+3\xi_s^*-2\xi_{\beta}$.
\item The event concerning the response defined as
$\mathcal{E}_y=\left\{||Y||^2<2c_y n^{1+\xi_y}\right\}$.
\end{enumerate}

The probability bounds for $\mathcal{E}_\lambda^c$ and $\mathcal{E}_\epsilon^c$ are already presented in inequalities \eqref{eq:elambda} and \eqref{eq:eepsilon}. And according to assumption (A4), we also have
\begin{equation}\label{eq:ey}
P\left(\mathcal{E}_y^c\right)\leq O\left(\exp\left(-c_Yn^{\xi_Y}\right)\right).
\end{equation}

Denoting $\mathcal{S}_{0}=\emptyset$ and $M_{\mathcal{S}_{0}}=I_n$, for $k\geq 1$, we define
\begin{align*}
\rho_k&=\min_{\bm{\beta}\in\mathbb{R}^{{\scriptstyle s}_{\scalebox{.7}{$\scriptscriptstyle k-1$}}+{\scriptstyle a}_{\scalebox{.7}{$\scriptscriptstyle k$}}}}\left[||Y-X_{\mathcal{S}_{k-1}\cup \mathcal{A}_{k}}\bm{\beta}||^2 + n\sum_{j=1}^{s_{k-1}+a_k}p_{\lambda_n}(|\beta_j|)\right]\\
&=\min_{\bm{\beta}\in\mathbb{R}^{{\scriptstyle s}_{\scalebox{.7}{$\scriptscriptstyle k$}}}}\left[||Y-X_{\mathcal{S}_{k}}\bm{\beta}||^2 + n\sum_{j=1}^{s_k}p_{\lambda_n}(|\beta_j|)\right].
\end{align*}

From the definition of $\rho_k$, we have
\begin{equation}\label{eq:thm3.1}
\rho_k\geq\min_{\bm{\beta}\in\mathbb{R}^{{\scriptstyle s}_{\scalebox{.7}{$\scriptscriptstyle k$}}}}||Y-X_{\mathcal{S}_{k}}\bm{\beta}||^2 =||M_{\mathcal{S}_{k}}Y||^2,
\end{equation}
and
\begin{equation}\label{eq:thm3.2}
\rho_k\leq\left.\left[||Y-X_{\mathcal{S}_{k-1}\cup \mathcal{A}_{k}}\bm{\beta}||^2 + n\sum_{j=1}^{s_{k-1}+a_k}p_{\lambda_n}(|\beta_j|)\right]\right|_{\bm{\beta}=0}\leq ||Y||^2.
\end{equation}

Notice that event $\mathcal{E}$ also indicates that $\mathcal{T}_{{k}}^c=\mathcal{T}-\mathcal{S}_{k}\neq\emptyset$ for all $1\leq k\leq \kappa$.
Then we evaluate the uniform lower bound of $\rho_k-\rho_{k+1}$ in events $\mathcal{E}$, $\mathcal{E}_\lambda$, $\mathcal{E}_\epsilon$ and $\mathcal{E}_y$. For $1\leq k\leq \kappa$, we define
\[i_{k+1}^*=\argmax_{i\in \mathcal{T}_{{k}}^c}(X_{i}^\top M_{\mathcal{S}_{k}}Y)^2.\]

\textbf{Case 1:} If $i_{k+1}^*\in\mathcal{A}_{k+1}$, we have
\begin{align*}
\rho_{k+1}&=\min_{\bm{\beta}\in\mathbb{R}^{{\scriptstyle s}_{\scalebox{.7}{$\scriptscriptstyle k$}}+{\scriptstyle a}_{\scalebox{.7}{$\scriptscriptstyle k+1$}}}}\left[||Y-X_{\mathcal{S}_{k}\cup \mathcal{A}_{k+1}}\bm{\beta}||^2 + n\sum_{j=1}^{s_{k}+a_{k+1}}p_{\lambda_n}(|\beta_j|)\right]\\
&\leq||M_{\mathcal{S}_{k}\cup\{i_{k+1}^*\}}Y||^2+n\sum_{j=1}^{s_{k}+1}p_{\lambda_n}(|\hat{\beta}_j^*|),
\end{align*}
where $\hat{\bm{\beta}}^*=(\hat{\beta}_1^*,\cdots,\hat{\beta}_{s_k+1}^*)^\top =(X_{\mathcal{S}_{k}\cup\{i_{k+1}^*\}}^\top X_{\mathcal{S}_{k}\cup\{i_{k+1}^*\}})^{-1}X_{\mathcal{S}_{k}\cup\{i_{k+1}^*\}}^\top Y$ is the OLS estimate of the coefficient in the regression between $Y$ and $X_{\mathcal{S}_{k}\cup\{i_{k+1}^*\}}$. Therefore, from inequality \eqref{eq:thm3.1}, we have
\begin{align}
\rho_k-\rho_{k+1}&\geq ||M_{\mathcal{S}_{k}}Y||^2-||M_{\mathcal{S}_{k}\cup\{i_{k+1}^*\}}Y||^2 - n\sum_{j=1}^{s_{k}+1}p_{\lambda_n}(|\hat{\beta}_j^*|)\nonumber\\
&\geq \frac{(X_{i_{k+1}^*}^\top M_{\mathcal{S}_{k}}Y)^2}{X_{i_{k+1}^*}^\top X_{i_{k+1}^*}}-n\lambda_n\sum_{j=1}^{s_{k}+1}|\hat{\beta}_j^*|,\label{eq:thm3.3}
\end{align}
where the last inequality comes from Corollary \ref{cor:2to1}. For the first term on the right-hand side of inequality \eqref{eq:thm3.3}, according to inequality \eqref{eq:marglb}, we have
\[\frac{(X_{i_{k+1}^*}^\top M_{\mathcal{S}_{k}}Y)^2}{X_{i_{k+1}^*}^\top X_{i_{k+1}^*}}\geq \frac{c_\beta^2\tau_{\min}^2}{4\tau_{\max}}n^{1-2\xi_\beta}.\] 
When it comes to the second term, for any $i\not\in\mathcal{S}_k$, we compute
\[\hat{\bm{\beta}}=(\hat{\beta}_1,\cdots,\hat{\beta}_{s_k+1})^\top =(X_{\mathcal{S}_{k}\cup\{i\}}^\top X_{\mathcal{S}_{k}\cup\{i\}})^{-1}X_{\mathcal{S}_{k}\cup\{i\}}^\top Y.\]
Then, when events $\mathcal{E}_\lambda$ and $\mathcal{E}_y$ hold, we obtain that
\begin{align}
\left(\sum_{j=1}^{s_{k}+1}|\hat\beta_j|\right)^2 &\leq (s_k+1) \cdot \sum_{j=1}^{s_k+1}\hat\beta_j^2\nonumber\\
&= (s_k+1) \cdot Y^\top X_{\mathcal{S}_{k}\cup \{i\}}(X_{\mathcal{S}_{k}\cup \{i\}}^\top  X_{\mathcal{S}_{k}\cup \{i\}})^{-2} X_{\mathcal{S}_{k}\cup\{ i\}}^\top Y\nonumber\\
&\leq (s_k+1)\cdot Y^\top H_{\mathcal{S}_{k}\cup \{i\}}Y/{\lambda_{\text{min}}(X_{\mathcal{S}_{k}\cup \{i\}}^\top  X_{\mathcal{S}_{k}\cup \{i\}})}\notag\\
&\leq{s_\kappa^*\cdot ||Y||^2}/{\lambda_{\text{min}}(X_{\mathcal{S}_{k}\cup \{i\}}^\top  X_{\mathcal{S}_{k}\cup\{ i\}})}\nonumber\\
&\leq 2C_s^*n^{\xi_s^*}\cdot 2c_y n^{1+\xi_y}/(n\tau_{\min})\nonumber\\
&= c^*n^{\xi_s^*+\xi_y},\label{eq:thm3.4}
\end{align}
where $c^*=4c_y C_s^* /\tau_{\min}$. Therefore, inequality \eqref{eq:thm3.3} can be followed by
\begin{align}
\rho_k-\rho_{k+1}
&\geq \frac{(X_{i_{k+1}^*}^\top M_{\mathcal{S}_{k}}Y)^2}{X_{i_{k+1}^*}^\top X_{i_{k+1}^*}}-n\lambda_n\sum_{j=1}^{s_{k}+1}|\hat{\beta}_j^*|\nonumber\\
&\geq \frac{c_\beta^2\tau_{\min}^2}{4\tau_{\max}}n^{1-2\xi_\beta}-\sqrt{c^*}\cdot\lambda_n n^{1+(\xi_s^*+\xi_y)/2}\nonumber\\
&=\frac{c_\beta^2\tau_{\min}^2}{4\tau_{\max}}n^{1-2\xi_\beta}\left(1-\frac{4\sqrt{c^*}\tau_{\max}}{c_\beta^2\tau_{\min}^{2}}\lambda_nn^{2\xi_\beta+(\xi_s^*+\xi_y)/2}\right).\label{eq:thm3.5}
\end{align}
Consequently, if we choose a tuning parameter satisfying that
\[\lambda_n\leq \frac{c_\beta^2\tau_{\min}^{2}}{8\sqrt{c^*}\tau_{\max}}\cdot n^{-2\xi_\beta-(\xi_s^*+\xi_y)/2},\] 
then inequality \eqref{eq:thm3.5} is followed by
\begin{equation}\label{eq:lb3.1}
\rho_k-\rho_{k+1}\geq \frac{c_\beta^2\tau_{\min}^2}{8\tau_{\max}}n^{1-2\xi_\beta}.
\end{equation}

\textbf{Case 2:} If $i_{k+1}^*\not\in\mathcal{A}_{k+1}$, let $i_{k+1}^{**}$ denote an arbitrary element in $\mathcal{A}_{k+1}$ and compute
\[\hat{\bm{\beta}}^{**}=(\hat{\beta}_1^{**},\cdots,\hat{\beta}_{s_k+1}^{**})^\top =(X_{\mathcal{S}_{k}\cup\{i_{k+1}^{**}\}}^\top X_{\mathcal{S}_{k}\cup\{i_{k+1}^{**}\}})^{-1}X_{\mathcal{S}_{k}\cup\{i_{k+1}^{**}\}}^\top Y.\]
Thus, similar to inequality \eqref{eq:thm3.3}, we can achieve
\begin{equation}\label{eq:lb3.2.0}
\rho_k-\rho_{k+1}\geq \frac{(X_{i_{k+1}^{**}}^\top M_{\mathcal{S}_{k}}Y)^2}{X_{i_{k+1}^{**}}^\top X_{i_{k+1}^{**}}}-n\lambda_n\sum_{j=1}^{s_{k}+1}|\hat{\beta}_j^{**}|.
\end{equation}
Moreover, we notice that inequality \eqref{eq:thm3.4} also holds for $\hat{\bm{\beta}}^{**}$ under $\mathcal{E}_\lambda$ and $\mathcal{E}_y$, that is
\[\left(\sum_{j=1}^{s_{k}+1}|\hat\beta_j^{**}|\right)^2\leq c^*n^{\xi_s^*+\xi_y}.\]

\textbf{Screening Criterion 1:} When applying the Screening Criterion 1,  we have
\[(X_{i_{k+1}^{**}}^\top M_{\mathcal{S}_{k}}Y)^2\geq (X_{i_{k+1}^*}^\top M_{\mathcal{S}_{k}}Y)^2\geq\frac{c_\beta^2\tau_{\min}^2}{4}n^{2-2\xi_\beta}.\]
If we choose a tuning parameter satisfying $\lambda_n\leq\frac{c_\beta^2\tau_{\min}^{2}}{8\sqrt{c^*}\tau_{\max}}n^{-2\xi_\beta-(\xi_s^*+\xi_y)/2}$, inequality \eqref{eq:lb3.2.0} is followed by 
\begin{equation}\label{eq:lb3.2.1}
\rho_k-\rho_{k+1}\geq\frac{c_\beta^2\tau_{\min}^2}{4\tau_{\max}}n^{1-2\xi_\beta}-\sqrt{c^*}\cdot\lambda_n n^{1+(\xi_s^*+\xi_y)/2}\geq \frac{c_\beta^2\tau_{\min}^2}{8\tau_{\max}}n^{1-2\xi_\beta}.
\end{equation}

\textbf{Screening Criterion 2:} In the Screening Criterion 2, from inequality \eqref{eq:xmlb2}, we have
\[
(X_{i_{k+1}^{**}}^\top M_{\mathcal{S}_{k}}Y)^2\geq \frac{\tau_{\min}}{\tau_{\max}}\cdot(X_{i_{k+1}^{*}}^\top M_{\mathcal{S}_{k}}Y)^2\geq\frac{c_\beta^2\tau_{\min}^3}{4\tau_{\max}}n^{2-2\xi_\beta}.
\]
If we choose $\lambda_n\leq\frac{c_\beta^2\tau_{\min}^{3}}{8\sqrt{c^*}\tau_{\max}^2}n^{-2\xi_\beta-(\xi_s^*+\xi_y)/2}$, we obtain that
\begin{equation}\label{eq:lb3.2.2}
\rho_k-\rho_{k+1}\geq\frac{c_\beta^2\tau_{\min}^3}{4\tau_{\max}^2}n^{1-2\xi_\beta}-\sqrt{c^*}\cdot\lambda_n n^{1+(\xi_s^*+\xi_y)/2}\geq \frac{c_\beta^2\tau_{\min}^3}{8\tau_{\max}^2}n^{1-2\xi_\beta}.
\end{equation}

\textbf{Screening Criterion 3:} When applying the Screening Criterion 3, from inequality \eqref{eq:xmlb3}, we have
\[
(X_{i_{k+1}^{**}}^\top M_{\mathcal{S}_{k}}Y)^2\geq \frac{\tau_{\min}^2}{\tau_{\max}^2}\cdot(X_{i_{k+1}^{*}}^\top M_{\mathcal{S}_{k}}Y)^2\geq\frac{c_\beta^2\tau_{\min}^4}{4\tau_{\max}^2}n^{2-2\xi_\beta}.
\]
If we choose $\lambda_n\leq\frac{c_\beta^2\tau_{\min}^{4}}{8\sqrt{c^*}\tau_{\max}^3}n^{-2\xi_\beta-(\xi_s^*+\xi_y)/2}$, inequality \eqref{eq:lb3.2.0} is followed by 
\begin{equation}\label{eq:lb3.2.3}
\rho_k-\rho_{k+1}\geq\frac{c_\beta^2\tau_{\min}^4}{4\tau_{\max}^3}n^{1-2\xi_\beta}-\sqrt{c^*}\cdot\lambda_n n^{1+(\xi_s^*+\xi_y)/2}\geq \frac{c_\beta^2\tau_{\min}^4}{8\tau_{\max}^3}n^{1-2\xi_\beta}.
\end{equation}

Combining the results in inequalities \eqref{eq:lb3.1}, \eqref{eq:lb3.2.1}, \eqref{eq:lb3.2.2} and \eqref{eq:lb3.2.3}, when events $\mathcal{E}$, $\mathcal{E}_\lambda$, $\mathcal{E}_\epsilon$ and $\mathcal{E}_y$ hold with a tuning parameter satisfying that $\lambda_n\leq\frac{c_\beta^2\tau_{\min}^{4}}{8\sqrt{c^*}\tau_{\max}^3}n^{-2\xi_\beta-(\xi_s^*+\xi_y)/2}$, then we have
\[
\rho_k-\rho_{k+1}\geq \frac{c_\beta^2\tau_{\min}^4}{8\tau_{\max}^3}n^{1-2\xi_\beta}.
\]
Consequently, events $\mathcal{E}$, $\mathcal{E}_\lambda$, $\mathcal{E}_\epsilon$ and $\mathcal{E}_y$ together imply that
\[
\rho_1\geq\rho_1-\rho_{\kappa+1}=\sum_{k=1}^\kappa (\rho_k-\rho_{k+1})\geq \kappa\cdot \frac{c_\beta^2\tau_{\min}^4}{8\tau_{\max}^3}n^{1-2\xi_\beta}\geq 2c_y n^{1+\xi_y}.
\]
From inequality \eqref{eq:thm3.2}, we know that $\rho_1\leq||Y||^2$. Then, under assumption (A4), we have
\[
P\left(\rho_1\geq 2c_y n^{1+\xi_y}\right)\leq P\left(\frac{1}{n}||Y||^2\geq 2\sigma_y^2\right)\leq O\left(\exp\left(-c_Yn^{\xi_Y}\right)\right),
\]
which indicates that
\[P\left(\mathcal{E}\cap\mathcal{E}_\lambda\cap\mathcal{E}_\epsilon\cap\mathcal{E}_y\right)\leq O\left(\exp\left(-c_Yn^{\xi_Y}\right)\right).\]
Finally, from probability bounds \eqref{eq:elambda}, \eqref{eq:eepsilon} and \eqref{eq:ey}, we obtain that
\begin{align*}
&~~~~P\left(\mathcal{T}\not\subset\mathcal{S}_k\text{ for all }1\leq k\leq \kappa\right)=P\left(\mathcal{E}\right)\\
&\leq P\left(\mathcal{E}\cap\mathcal{E}_\lambda\cap\mathcal{E}_\epsilon\cap\mathcal{E}_y\right)+P\left(\mathcal{E}_\lambda^c\right)+P\left(\mathcal{E}_\epsilon^c\right)+P\left(\mathcal{E}_y^c\right)\\
&\leq O\left(\exp\left(-c_Yn^{\xi_Y}\right)\right)+O\left(\exp\left(-c_\epsilon n^{\xi_\epsilon}/2\right)\right)+O\left(\exp\left(-c_s^* n^{1-2\xi_s^*}\right)\right)\\
&\leq O\left(\exp\left(-c n^{\xi}\right)\right),
\end{align*}
where $c$ is some positive constant and $\xi=\min\{{\xi_Y}, \xi_p+3\xi_s^*-2\xi_{\beta}, {1-2\xi_s^*}\}$ with $\xi_s^*=\xi_t\vee \xi_s$.
\end{proof}

\section{Discussion}
In this paper, we prove the sure screening properties of three types of iterative screening algorithms under reasonable assumptions, where the sure screening properties of many classical screening methods, such as FR, ISIS and Van-ISIS, can be achieved directly from our results. Currently, we only consider iterative algorithms on linear models. In the future work, we will investigate the sure screening properties of iterative screening methods applying general loss functions.

\appendix
\section{Proof of preliminary results}
\begin{proof}[Proof of Proposition \ref{prop:rsslb}] Recall that
\[
H_{\mathcal S\cup \mathcal A}= \begin{bmatrix}
{X}_{\mathcal S} & {X}_{\mathcal A}
\end{bmatrix}
\begin{bmatrix}
X_{\mathcal S}^\top X_{\mathcal S} & X_{\mathcal S}^\top X_{\mathcal A} \\  
X_{\mathcal A}^\top X_{\mathcal S} & X_{\mathcal A}^\top X_{\mathcal A}  
\end{bmatrix} ^{-1}                 
\begin{bmatrix}     
    X_{\mathcal S} & 
    X_{\mathcal A} 
\end{bmatrix}^\top .\]

Notice that $X_{\mathcal A}^\top M_{\mathcal S}X_{\mathcal A}$ is invertible when $[X_{\mathcal S}, X_{\mathcal A}]$ is of full column rank. Then, denoting $D_{1}=(X_{\mathcal S}^\top X_{\mathcal S})^{-1}$, $D_2=X_{\mathcal S}^\top X_{\mathcal A}$, $D_3=X_{\mathcal A}^\top X_{\mathcal S}$ and $D_4=(X_{\mathcal A}^\top M_{\mathcal S}X_{\mathcal A})^{-1}$, by the blockwise inverse formula \citep{binv}, we have
\[
\begin{bmatrix}
X_{\mathcal S}^\top X_{\mathcal S}& X_{\mathcal S}^\top X_{\mathcal A} \\  
X_{\mathcal A}^\top X_{\mathcal S}& X_{\mathcal A}^\top X_{\mathcal A} 
\end{bmatrix} ^{-1}
=
\begin{bmatrix}
D_1 +D_1D_2D_4D_3D_1 & -D_1D_2D_4 \\  
-D_4D_3D_1 & D_4  
\end{bmatrix}.
\]
Consequently, we have
\begin{align*}
&~~~~\begin{bmatrix}
{X}_{\mathcal S} & {X}_{\mathcal A}
\end{bmatrix}
\begin{bmatrix}
X_{\mathcal S}^\top X_{\mathcal S} & X_{\mathcal S}^\top X_{\mathcal A} \\  
X_{\mathcal A}^\top X_{\mathcal S} & X_{\mathcal A}^\top X_{\mathcal A}  
\end{bmatrix} ^{-1}                 
\begin{bmatrix}     
    X_{\mathcal S} &
    X_{\mathcal A} 
\end{bmatrix}^\top \\
&=H_{\mathcal S}-M_{\mathcal S}X_{\mathcal A}D_4X_{\mathcal A}^\top H_{\mathcal S}+M_{\mathcal S}X_{\mathcal A}D_4X_{\mathcal A}^\top \\
&=H_{\mathcal S}+M_{\mathcal S}X_{\mathcal A}(X_{\mathcal A}^\top M_{\mathcal S}X_{\mathcal A})^{-1}X_{\mathcal A}^\top M_{\mathcal S}.
\end{align*}
Therefore, we achieve that
\[M_{\mathcal S} -M_{\mathcal S\cup \mathcal A}= H_{\mathcal S\cup \mathcal A}-H_{\mathcal S}=M_{\mathcal S}X_{\mathcal A}(X_{\mathcal A}^\top M_{\mathcal S}X_{\mathcal A})^{-1}X_{\mathcal A}^\top M_{\mathcal S}.\]
Consequently, we have
\begin{align}
||{M}_{\mathcal S}{Y}||^2-||{M}_{\mathcal S\cup \mathcal A}{Y}||^2&=Y^\top M_{\mathcal S}X_{\mathcal A}(X_{\mathcal A}^\top M_{\mathcal S}X_{\mathcal A})^{-1}X_{\mathcal A}^\top M_{\mathcal S}Y\label{eq:p1.1}\\
&\geq \lambda_{\text{min}}\{(X_{\mathcal A}^\top M_{\mathcal S}X_{\mathcal A})^{-1}\}\cdot ||X_{\mathcal A}^\top M_{\mathcal S}Y||^2\nonumber\\
&\geq\sum_{i\in {\mathcal A}}({X}^\top _i{M}_{\mathcal S}{Y})^2/\lambda_{\text{max}}({X}^\top _{\mathcal A}M_{\mathcal S}{X}_{\mathcal A}).\nonumber
\end{align}
\end{proof}
\begin{proof}[Proof of Corollary \ref{cor:2to1}]
The conclusion can be achieved directly from inequality \eqref{eq:p1.1} by setting $\mathcal{A}=\{i\}$.
\end{proof}
\begin{proof}[Proof of Proposition \ref{prop:3to1}] 
For any $i\not\in \mathcal S$, $\hat\beta_i$ can be regarded as the last entry in the OLS estimate
\[(X_{\mathcal S\cup \{i\}}^\top X_{\mathcal S\cup \{i\}})^{-1}X_{\mathcal S\cup \{i\}}^\top Y.\]
Therefore, according to the blockwise inverse formula \citep{binv}, the estimate $\hat{\beta}_i$ can be expressed as
\begin{align}
\hat\beta_i&=[0,\cdots,1](X_{\mathcal S\cup \{i\}}^\top X_{\mathcal S\cup \{i\}})^{-1}X_{\mathcal S\cup \{i\}}^\top Y\nonumber\\
&=[-(X_i^\top M_{\mathcal S}X_i)^{-1}X_i^\top X_{\mathcal S}(X_{\mathcal S}^\top X_{\mathcal S})^{-1}X_{\mathcal S}^\top +(X_i^\top M_{\mathcal S}X_i)^{-1}X_i^\top ]Y\nonumber\\
&=(X_i^\top M_{\mathcal S}X_i)^{-1}X_i^\top M_{\mathcal S}Y.
\end{align}
\end{proof}
\begin{proof}[Proof of Proposition \ref{prop:lb1}]
Recall that $M_{\mathcal S}X_{\mathcal S}=0$. Then for any $i\in {\mathcal T}_{\mathcal S}^c$, we have 
\begin{equation}\label{p2.1}
{X}^\top _i {M}_{\mathcal S}{Y}={X}^\top _i {M}_{\mathcal S}{X}_{\mathcal{T}_{\mathcal S}^c}{\bm\beta}_{\mathcal{T}_{\mathcal S}^c}+{X}^\top _i{M}_{\mathcal S}{\bm \epsilon}.
\end{equation}
For the first term in the right-hand side of equation \eqref{p2.1}, we have
\begin{align}
\max\limits_{i\in {\mathcal T}_{\mathcal S}^c}({X}^\top _i{M}_{\mathcal S}{X}_{\mathcal{T}_{\mathcal S}^c}{\bm \beta}_{\mathcal{T}_{\mathcal S}^c})^2&\geq |\mathcal{T}_{\mathcal S}^c|^{-1}\cdot \sum_{i\in \mathcal{T}_{\mathcal S}^c}({X}^\top _i{M}_{\mathcal S}{X}_{\mathcal{T}_{\mathcal S}^c}{\bm \beta}_{\mathcal{T}_{\mathcal S}^c})^2\nonumber\\
&=|\mathcal{T}_{\mathcal S}^c|^{-1}\cdot||{X}^\top _{\mathcal {T}_{\mathcal S}^c}{M}_{\mathcal S}{X}_{\mathcal{T}_{\mathcal S}^c}\bm{\beta}_{\mathcal{T}_\mathcal{S}^c}||^2\nonumber\\
&\geq|\mathcal{T}_{\mathcal S}^c|^{-1}\cdot||\bm{\beta}_{\mathcal{T}_{\mathcal S}^c}^\top \cdot{X}^\top _{\mathcal{T}_{\mathcal S}^c}{M}_{\mathcal S}{X}_{\mathcal{T}_{\mathcal S}^c}\bm{\beta}_{\mathcal{T}_{\mathcal S}^c}||^2/||\bm{\beta}_{\mathcal{T}_{\mathcal S}^c}||^2\nonumber\\
&= |\mathcal{T}_{\mathcal S}^c|^{-1}\cdot||{M}_{\mathcal S}{X}_{\mathcal{T}_{\mathcal S}^c}\bm{\beta}_{\mathcal{T}_{\mathcal S}^c}||^4/||\bm{\beta}_{\mathcal{T}_{\mathcal S}^c}||^2.\label{p2.2}
\end{align}
Meanwhile, under the condition that ${X}_{\mathcal S^*}$ is of full column rank, we obtain
\begin{align}
||{M}_{\mathcal S}{X}_{{\mathcal T}_{\mathcal S}^c}\bm{\beta}_{{\mathcal T}_{\mathcal S}^c}||^2
&=||[{X}_{{\mathcal T}_{\mathcal S}^c}-{X}_{\mathcal{S}}({X}^\top _{\mathcal{S}}{X}_{\mathcal{S}})^{-1}{X}^\top _{\mathcal{S}}{X}_{\mathcal{T}_{\mathcal S}^c}]\bm{\beta}_{\mathcal{T}_{\mathcal S}^c}||^2\nonumber\\
&=||[{X}_{\mathcal S},{X}_{\mathcal{T}_{\mathcal S}^c}][-G^\top ,I]^\top \bm{\beta}_{\mathcal{T}_{\mathcal S}^c}||^2\nonumber\\
&\geq \lambda_{\text{min}}({X}^\top _{\mathcal S^*}{X}_{\mathcal S^*})\cdot ||[-G^\top ,I]^\top \bm{\beta}_{\mathcal{T}_{\mathcal S}^c}||^2\nonumber\\
&\geq \lambda_{\text{min}}({X}^\top _{\mathcal S^*}{X}_{\mathcal S^*})\cdot||\bm{\beta}_{\mathcal{T}_{\mathcal S}^c}||^2,\label{p2.3}
\end{align}
where $G=({X}^\top _{\mathcal{S}}{X}_{\mathcal{S}})^{-1}{X}^\top _{\mathcal{S}}{X}_{\mathcal{T}_{\mathcal S}^c}$. Then combining inequalities \eqref{p2.2} and \eqref{p2.3}, we have
\begin{align}
\max\limits_{i\in {\mathcal T}_{\mathcal S}^c}({X}^\top _i{M}_{\mathcal S}{X}_{\mathcal{T}_{\mathcal S}^c}\bm{\beta}_{\mathcal{T}_{\mathcal S}^c})^2&\geq |\mathcal{T}_{\mathcal S}^c|^{-1}\cdot\lambda^2_{\text{min}}({X}^\top _{\mathcal S^*}{X}_{\mathcal S^*})\cdot||\bm{\beta}_{\mathcal{T}_{\mathcal S}^c}||^2\nonumber\\
&\geq \lambda^2_{\text{min}}({X}^\top _{\mathcal S^*}{X}_{\mathcal S^*})\cdot\beta^2_{\text{min}}.\label{p2.4}
\end{align}
With the fact that ${X}^\top _{i}{X}_{i}\geq ||{X}^\top _i{M}_{\mathcal S}||^2$, equation \eqref{p2.1} is followed by
\begin{align*}
\max\limits_{i\in {\mathcal T}_{\mathcal S}^c}({X}^\top _i{M}_{\mathcal S}{Y})^2
&\geq \frac{1}{2}\max\limits_{i\in {\mathcal T}_{\mathcal S}^c}({X}^\top _i{M}_{\mathcal S}{X}_{\mathcal{T}_{\mathcal S}^c}\bm{\beta}_{\mathcal{T}_{\mathcal S}^c})^2-\max\limits_{i\in {\mathcal T}_{S}^c}({X}^\top _i{M}_{\mathcal S}{\bm\epsilon})^2\\
&\geq \frac{\beta^2_{\text{min}}}{2}\cdot\lambda^2_{\text{min}}({X}^\top _{\mathcal S^*}{X}_{\mathcal S^*})-\max\limits_{i\in {\mathcal  T}_{S}^c}({X}^\top _{i}{X}_{i})\cdot\max\limits_{i\in \mathcal{T}_{\mathcal S}^c}\left(\frac{{X}^\top _i{M}_{\mathcal S}{\bm\epsilon}}{||{X}^\top _i{M}_{\mathcal S}||}\right)^2.
\end{align*}
\end{proof}

\section*{Acknowledgements}
We would like to thank Professor Hansheng Wang for his valuable comments on some technical details in our proof.

\bibliography{aosref}

\begin{thebibliography}{13}
\providecommand{\natexlab}[1]{#1}
\providecommand{\url}[1]{\texttt{#1}}
\expandafter\ifx\csname urlstyle\endcsname\relax
  \providecommand{\doi}[1]{doi: #1}\else
  \providecommand{\doi}{doi: \begingroup \urlstyle{rm}\Url}\fi

\bibitem[Bernstein(2009)]{binv}
Dennis~S. Bernstein.
\newblock \emph{Matrix mathematics. Theory, facts, and formulas}.
\newblock Princeton University Press, Princeton, NJ, second edition, 2009.

\bibitem[Fan and Li(2001)]{scad}
Jianqing Fan and Runze Li.
\newblock Variable selection via nonconcave penalized likelihood and its oracle
  properties.
\newblock \emph{J. Amer. Statist. Assoc.}, 96\penalty0 (456):\penalty0
  1348--1360, 2001.

\bibitem[Fan and Lv(2008)]{sis}
Jianqing Fan and Jinchi Lv.
\newblock Sure independence screening for ultrahigh dimensional feature space.
\newblock \emph{J. R. Stat. Soc. Ser. B Stat. Methodol.}, 70\penalty0
  (5):\penalty0 849--911, 2008.

\bibitem[Fan et~al.(2009)Fan, Samworth, and Wu]{vanisis}
Jianqing Fan, Richard Samworth, and Yichao Wu.
\newblock Ultrahigh dimensional feature selection: beyond the linear model.
\newblock \emph{J. Mach. Learn. Res.}, 10:\penalty0 2013--2038, 2009.

\bibitem[Fang et~al.(2015)Fang, Qin, Zhang, Wang, Wang, and Zheng]{isisapp}
Yun Fang, Yufang Qin, Naiqian Zhang, Jun Wang, Haiyun Wang, and Xiaoqi Zheng.
\newblock {DISIS}: prediction of drug response through an iterative sure
  independence screening.
\newblock \emph{PloS ONE}, 10\penalty0 (3):\penalty0 e0120408, 2015.

\bibitem[Hoeffding(1963)]{hoeffding}
Wassily Hoeffding.
\newblock Probability inequalities for sums of bounded random variables.
\newblock \emph{J. Amer. Statist. Assoc.}, 58:\penalty0 13--30, 1963.

\bibitem[Saldana and Feng(2018)]{isisr}
Diego Saldana and Yang Feng.
\newblock {SIS}: an {R} package for sure independence screening in
  ultrahigh-dimensional statistical models.
\newblock \emph{Journal of Statistical Software, Articles}, 83\penalty0
  (2):\penalty0 1--25, 2018.

\bibitem[Tibshirani(1996)]{lasso}
Robert Tibshirani.
\newblock Regression shrinkage and selection via the lasso.
\newblock \emph{J. Roy. Statist. Soc. Ser. B}, 58\penalty0 (1):\penalty0
  267--288, 1996.

\bibitem[Vershynin(2010)]{ver}
Roman Vershynin.
\newblock Introduction to the non-asymptotic analysis of random matrices.
  preprint. available at
  \href{https://arxiv.org/abs/1011.3027}{arXiv:1011.3027}.
\newblock 2010.

\bibitem[Wang(2009)]{fr}
Hansheng Wang.
\newblock Forward regression for ultra-high dimensional variable screening.
\newblock \emph{J. Amer. Statist. Assoc.}, 104\penalty0 (488):\penalty0
  1512--1524, 2009.

\bibitem[Wang and Leng(2016)]{holp}
Xiangyu Wang and Chenlei Leng.
\newblock High dimensional ordinary least squares projection for screening
  variables.
\newblock \emph{J. R. Stat. Soc. Ser. B. Stat. Methodol.}, 78\penalty0
  (3):\penalty0 589--611, 2016.

\bibitem[Zou(2006)]{adalasso}
Hui Zou.
\newblock The adaptive lasso and its oracle properties.
\newblock \emph{J. Amer. Statist. Assoc.}, 101\penalty0 (476):\penalty0
  1418--1429, 2006.

\bibitem[Zou and Hastie(2005)]{elanet}
Hui Zou and Trevor Hastie.
\newblock Regularization and variable selection via the elastic net.
\newblock \emph{J. R. Stat. Soc. Ser. B Stat. Methodol.}, 67\penalty0
  (2):\penalty0 301--320, 2005.

\end{thebibliography}
\end{document}